\documentclass[10pt,reqno]{amsart}
\pdfoutput=1 
\usepackage[foot]{amsaddr}
\usepackage{etoolbox}
\usepackage[a4paper,margin=3cm]{geometry}
\usepackage[utf8]{inputenc} 
\usepackage[T1]{fontenc}    
\usepackage{hyperref}       
\usepackage{url}            
\usepackage{booktabs}       
\usepackage{amsfonts}       
\usepackage{nicefrac}       
\usepackage{microtype}      
\usepackage{graphicx}       
\usepackage{tikz}           
\usepackage{natbib}         
\usepackage[ruled,vlined]{algorithm2e}  
\usepackage{amsmath,amssymb,amsthm} 
\usepackage{dsfont}                 
\usepackage{xcolor}                 
\usepackage{enumitem}               
\usepackage{cancel}
\usepackage{hyperref}

\title{Efficient~online~algorithms~for fast-rate~regret~bounds~under~sparsity}

\author{Pierre Gaillard$^1$}
\address{$^1$ INRIA - Département d’informatique de l’ENS \\
  Ecole normale supérieure, CNRS, INRIA \\ 
  PSL Research University, 75005 Paris, France}
\author{Olivier Wintenberger$^2$}
\address{$^2$ Sorbonne Université, LPSM, Paris, France}

\email{pierre.gaillard@inria.fr}
\email{olivier.wintenberger@upmc.fr}

\patchcmd{\section}{\scshape}{\bfseries\scshape}{}{}
\makeatletter
\renewcommand{\@secnumfont}{\bfseries}
\makeatother

\makeatletter
\def\paragraph{\@startsection{paragraph}{4}%
  \z@\z@{-\fontdimen2\font}%
  {\normalfont\bfseries}}
\makeatother

\newtheorem{theorem}{Theorem}[section]
\newtheorem{proposition}[theorem]{Proposition}
\newtheorem{lemma}[theorem]{Lemma}

\newtheorem{definition}{Definition}[section]

\theoremstyle{remark}
\newtheorem{remark}{Remark}[section]

\newcommand{\1}{\ensuremath{\mathrm{1}\hspace{-.35em} \mathrm{1}}} 

\makeatletter
\newcommand{\generate}[4]{%
  \def\@tempa{#1} 
  \count@=`#3
  \loop
  \begingroup\lccode`?=\count@
  \lowercase{\endgroup\@namedef{\@tempa ?}{#2{?}}}%
  \ifnum\count@<`#4
  \advance\count@\@ne
  \repeat
}
\generate{c}{\mathcal}{A}{Z}
\newcommand{\R}{\mathbb{R}}

\newcommand{\E}{\mathbb{E}}

\renewcommand{\log}{\ln}

\renewcommand{\hat}{\widehat}
\renewcommand{\leq}{\leqslant}
\renewcommand{\le}{\leqslant}
\renewcommand{\geq}{\geqslant}

\renewcommand{\epsilon}{\varepsilon}
\renewcommand{\b}{\boldsymbol}
\newcommand{\indic}{\mathds{1}}
\def\mystrut(#1,#2){\vrule height #1pt depth #2pt width 0pt}

\DeclareMathOperator*{\argmin}{arg\,min}

\DeclareMathOperator{\Card}{Card}
\DeclareMathOperator{\Reg}{Reg}

\DeclareMathOperator{\sign}{sign}
\DeclareMathOperator{\Supp}{Supp}

\makeatletter
\newcommand{\pushright}[1]{\ifmeasuring@#1\else\omit\hfill$\displaystyle#1$\fi\ignorespaces}
\newcommand{\pushleft}[1]{\ifmeasuring@#1\else\omit$\displaystyle#1$\hfill\fi\ignorespaces}
\makeatother

\newcounter{saveenum}

\begin{document}

\begin{abstract}
We consider the online convex optimization problem. In the setting of arbitrary sequences and finite set of parameters, we establish a new fast-rate  quantile regret bound. Then we investigate the optimization into the $\ell_1$-ball by discretizing the parameter space. Our algorithm is projection free and we propose an efficient solution by restarting the algorithm on adaptive discretization grids. In the adversarial setting, we develop an algorithm that achieves several rates of convergence with different dependences on the sparsity of the objective. In the i.i.d. setting, we establish new risk bounds that are adaptive to the sparsity of the problem and to the regularity of the risk (ranging from a rate $\nicefrac{1}{\sqrt{T}}$ for general convex risk to $\nicefrac{1}{T}$ for strongly convex risk). These results generalize previous works on sparse online learning. They are obtained under a weak assumption on the risk (Łojasiewicz's assumption) that allows multiple optima which is crucial when dealing with degenerate situations.
\end{abstract}

\maketitle 

\section{Introduction}

We consider the following setting of online convex prediction. Let $(\ell_t:\R^d \to \R)_{t\geq 1}$ be a collection of random convex sub-differentiable loss functions sequentially observed.  At each time step $t \geq 1$, a learner forms a prediction $\smash{\hat \theta_{t-1}\in \R^d}$ based on past observations $\smash{\cF_{t-1} = \{\ell_1, \hat \theta_1,\dots,\ell_{t-1},\hat \theta_{t-1}\}}$. The learner aims at minimizing its average risk
\begin{equation}
  \label{eq:regret}
  R_T(\theta) := \frac{1}{T} \sum_{t=1}^T \E_{t-1}\big[\ell_t(\hat \theta_{t-1})\big] - \frac{1}{T} \sum_{t=1}^T \E_{t-1}\big[\ell_t(\theta)\big]  \qquad \text{where} \quad \E_{t-1} = \E[\,\cdot\,|\cF_{t-1}] \,,
\end{equation}
with respect to all $\theta$ in some reference set $\Theta \subseteq  \cB_1 := \{\theta \in \R^d: \|\theta\|_1 \leq 1\}$. By considering the Dirac masses, one obtains   $\ell_t = \E_{t-1}[\ell_t]$ and the average risk matches the definition $\smash{(\nicefrac{1}{T})\sum_{t=1}^T \big(\ell_t(\hat \theta_{t-1}) - \ell_t(\theta)\big)}$ of the average regret more commonly used in the online learning literature. We will first consider finite set $\Theta$. Then we will show how to extend the results to the unit $\ell_1$-ball $\cB_1$ providing sparsity guarantees for sparse $\theta \in \cB_1$.

\medskip
\paragraph*{Related work} 
The case of finite reference set $\Theta$ corresponds to the setting of prediction with expert advice (see Section~\ref{sec:experts} or \citep{Cesa-BianchiLugosi2006,freund1997decision,vovk1998game}), where a learner makes sequential predictions over a series of rounds with the help of $K$ experts. \cite{littlestone1994weighted} and \cite{vovk1990aggregating} introduced the exponentially weighted average algorithm (Hedge) which achieves the optimal rate of convergence $\cO(\nicefrac{1}{\sqrt{T}})$ for the average regret for general convex functions. Several works focused on improving the rate of convergence under nice properties of the loss or the data. For instance, Hedge ensures a rate $\cO(\nicefrac{1}{T})$ for exp-concave loss functions. We refer to \cite{van2015fast} for a thorough review of fast-rate type assumptions on the losses.  

The extension from finite reference sets to convex sets is natural. The seminal paper~\cite{KivinenWarmuth1997} introduced the Exponentiated Gradient algorithm (EG), a version of Hedge using gradient version of the losses. The latter guarantees a $\cO(\smash{\nicefrac{1}{\sqrt{T}}})$ average regret uniformly over the unit $\ell_1$-ball $\cB_1$. Another approach consists in projecting gradient descent steps (see~\cite{Zinkevich2003} for general convex set, \cite{DuchiEtAl2008} for the $\ell_1$-ball, or \cite{AgarwalNegahbanWainwright2012} for fast rates under sparsity). 

First works in i.i.d. online convex optimization under sparsity was done by~\cite{AgarwalNegahbanWainwright2012,GaillardWintenberger2017,Steinhardt2014} that obtained sparse rates of order $\smash{\tilde \cO(\|\theta^*\|_0\log d/T)}$\footnote{Throughout the paper $\lesssim$ denotes an approximative inequality which holds up to universal constants and $\tilde \cO$ denotes an asymptotic inequality up to logarithmic terms in $T$ and dependence on parameters not clarified.}. Their settings are very close to the one of \cite{bunea2007} used for studying the convergence properties of the LASSO batch procedure. Their methods differ; the one of \cite{Steinhardt2014} uses a $\ell_1$-penalized gradient descent whereas the one of \cite{AgarwalNegahbanWainwright2012} and \cite{GaillardWintenberger2017} are based on restarting a subroutine centered around the current estimate, on sessions of exponentially growing length. These works compete with the optima over $\R^d$ assumed to be (approximately in~\cite{AgarwalNegahbanWainwright2012}) sparse with a known $\ell_1$-bound. In contrast, we only compete here with optima over $\cB_1$ which are more likely to be sparse.

Little work was done on sparsity under adversarial data.  The papers \cite{langford2009sparse,Xiao2010,duchi2010composite} focus on providing sparse estimators with rates of order $\smash{\cO(\nicefrac{1}{\sqrt{T}})}$ or a linear dependency on the dimension $d$. Recent work (see \cite{foster2016online,kale2017adaptive} and references therein) considers the problem where the learner only observes a sparse subset of coordinates at each round. Though they also compare themselves with sparse parameters, they also suffer a bound larger than $\smash{\cO(\nicefrac{1}{\sqrt{T}})}$.  Fast rate sparse regret bounds involving $\|\theta\|_0$ were, to our knowledge, only obtained through non-efficient procedures (see \cite{Gerchinovitz2011} or \cite{rakhlin2015online}). 

\medskip
\paragraph*{Contributions and outline of the paper}

In this paper we focus on providing fast rate regret bounds involving the sparsity of the objective $\|\theta\|_0$.

In Section~\ref{sec:finite} we start with the finite case $\Theta = \{\theta_1,\dots,\theta_K\}$. We extend the results of~\cite{Wintenberger2014} and  \cite{vanErvenKoolen2015} under a weak version of exp-concavity, see Assumption~\ref{ass:a2}. We show in Theorem~\ref{thm:BOA_fastrate} that the Bernstein Online Aggregation (BOA) and Squint algorithms achieve a fast rate with high probability: i.e. $R_{T}(\theta) \leq \cO((\ln K)/T)$
for arbitrary data. The theorem also provides a quantile bound on the risk which improves the dependency on $K$ if many experts are performing well. 
This is the first quantile-like bound on the average risk that provides fast-rate with high probability. \citet{mehta2016fast} developed high-probability quantile bounds but it was degrading with an additional gap term.

In Section~\ref{sec:B1}, we consider the case $\Theta=\cB_1$. The standard reduction using the ``gradient trick'' of \cite{KivinenWarmuth1997}, looses the fast-rate guaranty obtained under Assumption~\ref{ass:a2}. Considering BOA on a discretization grid $\Theta_0$ of $\Theta$ and applying Theorem~\ref{thm:BOA_fastrate} yields optimal convergence rate under~\ref{ass:a2}. Yet, the complexity of the discretization is prohibitive. 
We thus investigate how an a-priori discretization grid $\Theta_0$ may be used to improve the regret bound. We provide in Theorem~\ref{thm:slowrate} a bound of the form $\smash{R_T(\theta) \leq \cO(D(\theta,\Theta_0)/{\sqrt{T}})}$ which we call \emph{accelerable}, i.e. the rate may decrease if $\smash{D(\theta,\Theta_0)}$ decreases with $T$. Here $D$ is a pseudo-metric that we call \emph{averaging accelerability} and $D(\theta,\Theta_0)$ is the distance of $\theta$ with $\Theta_0$ in this pseudo-metric. Our bound yields an oracle bound of the form $\smash{R_T(\theta) \leq \cO({\|\theta\|_1}/{\sqrt{T}})}$ 
which was recently studied by~\cite{Foster2017}. The following sections \ref{sec:adv} and \ref{sec:acc} build the grid $\Theta_0$ adaptively in order to ensure a small regret under a sparsity scenario: Section~\ref{sec:adv} in the adversarial setting and Section~\ref{sec:acc} for i.i.d. losses. 

In Section~\ref{sec:adv}, we work under the strong convexity assumption on the losses in the adversarial setting. Using a doubling trick, we show that including sparse versions of the leader of the last session in $\Theta_0$ is enough to ensure that  $\smash{R_T(\theta) \leq \tilde \cO\big(({\sqrt{d\|\theta\|_0}}/{T})\wedge({\sqrt{\|\theta\|_0}}/T^{3/4})\big)}$ for all $\smash{\theta \in \cB_{1}}$. The rate is faster than the usual rate of convergence $\smash{\tilde \cO(d/T)}$ obtained by online gradient descent or online newton step \cite{hazan2007}. The gain $\smash{\sqrt{\|\theta\|_0/d}\wedge\sqrt{\|\theta\|_0/T}}$ is significant for sparse parameters $\theta$. The numerical and space complexities of the algorithm, called BOA+, are $\smash{\tilde \cO(dT)}$. Notice that the rate can be decreased to $\smash{\tilde \cO(d_0/T)}$ whenever the leaders and the parameter $\theta$ are $d_0$-sparse. This favorable case is not likely to happen in the adversarial setting but do happen in the i.i.d. setting treated in Section~\ref{sec:acc}.

A new difficulty raises in the i.i.d. setting: we accept only assumptions on the risk $\E[\ell_t]$ and not on the losses $\ell_t$. To do so, we need to enrich the grid $\Theta_0$ with good approximations of the optima of the risk $\E[\ell_t]$. However, the risk is not observed and the minimizer of the empirical risk (the leader) suffer a rate of convergence linear in $d$. Thus, we develop another algorithm, called SABOA, that sequentially enriches $\Theta_0$ by averaging the estimations of the algorithms on the last session. We extend the setting of strong convexity on $\R^d$ of the preceding results of~\cite{Steinhardt2014,GaillardWintenberger2017,AgarwalNegahbanWainwright2012} to the weaker Łojasiewicz's assumption~\ref{ass:Lojasiewicz} on the $\ell_1$-ball only. The latter was introduced by \cite{Loja63,Loja93} and states that there exist $\beta >0$ and $\mu>0$ such that for all $\theta \in \cB_1$, it exists a minimizer  $\theta^*$ of the risk over $\cB_1$ satisfying 
\[
  \mu \big\|\theta - \theta^*\big\|_2^2 \leq \E[\ell_t(\theta)-\ell_t(\theta^*)]^\beta \,.
\]
The Łojasiewicz's assumption depends on a parameter $\beta \in [0,1]$ that ranges from general convex functions ($\beta=0$) to strongly convex functions ($\beta=1$). Under this condition our algorithm achieves a fast rate upper-bound on the average risk of order $\smash{\tilde \cO((\|\theta^*\|_0\log(d)/T)^{1/(2-\beta)})}$ when the optimal parameters have $\ell_1$-norm bounded by $c<1$. When some optimal parameters $\theta^*$ lie on the border of the ball, the bound suffers an additional factor $\|\theta^*\|_0$.  Łojasiewicz's Assumption~\ref{ass:Lojasiewicz} also allows multiple optima which is crucial when we are dealing with degenerated collinear design (allowing zero eigenvalues in the Gram matrix). The complexity of the algorithm, called SABOA, is $\smash{\tilde \cO(dT)}$ and it is fully adaptive to all parameters except for the Lipschitz constant. 

To summarize our contributions, we provide
\begin{itemize}[label={-},nosep,topsep=-5pt,parsep=0pt,itemsep=0pt,leftmargin=15pt]
  \item the first hight-probability quantile bound achieving a fast rate (Theorem~\ref{thm:BOA_fastrate});
  \item a new bound on $R_T(\theta)$ that is small whenever $\theta$ is close to a grid provided in hindsight (Thm.~\ref{thm:slowrate});
  \item two efficient algorithms with improved average risks when $\theta$ is sparse in the adversarial setting with strongly convex losses (BOA+, Thm.~\ref{thm:adversarial}) and in the i.i.d. setting with Łojasiewicz's assumption (SABOA, Thm.~\ref{thm:acceleratedBOA}). 
\end{itemize}

\section{Finite reference set}
\label{sec:finite}
In this section, we focus on finite reference set $\Theta := \{\theta_1,\dots,\theta_K\} \subset \cB_1$. This is the case of the setting of prediction with expert advice presented in Section~\ref{sec:experts}.
We will consider the following two assumptions on the loss:
\begin{enumerate}[label={(A\arabic*)},topsep=-5pt,parsep=2pt,itemsep=0pt]

  \item\label{ass:a1}\emph{Lipschitz loss\footnote{Throughout the paper, we assume that  the Lipschitz constant $G $ in~\ref{ass:a1} is known. It can be calibrated online with standard tricks such as the doubling trick (see~\cite{CesaBianchiMansourStoltz2007} for instance) under sub-Gaussian conditions.}}: $\nabla \ell_t$ are sub-differential and for all $t \geq 1$, $\max_{\theta \in \cB_1}  \big\|\nabla \ell_t\big\|_\infty\leq G$.

  \item\label{ass:a2}\emph{Weak exp-concavity:} There exist  $\alpha >0$ and $\beta \in [0,1]$ such that for all $t \geq 1$, for all $\theta_1,\theta_2 \in \cB_1$, almost surely 
  \[
    \E_{t-1}\!\big[\ell_t(\theta_1) - \ell_t(\theta_2)\!\big] \leq \E_{t-1}\big[ \nabla \ell_t(\theta_1)^\top (\theta_1 - \theta_2)\big] \\
    - \E_{t-1} \!\Big[\!\left(\alpha \big(\nabla \ell_t(\theta_1)^\top(\theta_1 - \theta_2)\big)^{2} \right)^{1/\beta} \!\Big].
  \]
    \setcounter{saveenum}{\value{enumi}}
\end{enumerate}

\smallskip
For convex losses $(\ell_t)$, Assumption~\ref{ass:a2} is satisfied with $\beta = 0$ and $\alpha < G^{-2}$.  Fast rates are obtained for $\beta >0$. It is worth pointing out that Assumption \ref{ass:a2} is weak even in the strongest case $\beta=1$. It is implied by several common assumptions such as:
\begin{itemize}[label={--},topsep=-3pt,leftmargin=15pt, itemsep=0pt,parsep=2pt]
  \item \emph{Strong convexity of the risk}: under the boundedness of the gradients, assumption \ref{ass:a2} with $\alpha =\mu/(2G ^2)$ is implied by the $\mu$-strong convexity of the risks $(\E_{t-1}[\ell_t])$.
  \item \emph{Exp-concavity of the loss}: Lemma 4.2, \citet{Hazan2016} states that \ref{ass:a2} with $\alpha \leq \frac{1}{4}\min\{\frac{1}{8G},\kappa\}$ is implied by $\kappa$-exp-concavity of the loss functions $(\ell_t)$. Our assumption is slightly weaker since its  needs to hold in conditional expectation only.

\end{itemize}

\subsection{Fast-rate quantile bound with high probability}
\label{sec:fastrate_expert}

For prediction with $K\geq 1$ expert advice, \cite{Wintenberger2014} showed that a fast rate $\cO\big((\ln K)/T\big)$ can be obtained by the BOA algorithm under the LIST condition (i.e., Lipschitz and strongly convex losses) and i.i.d. estimators. Here, we show that Assumption~\ref{ass:a2} is enough.
By using the Squint algorithm of \cite{vanErvenKoolen2015} (see Algorithm~\ref{alg:BOA}), we also replace the dependency on the total number of experts with a quantile bound. The latter is smaller when many experts perform well. Note that Algorithm~\ref{alg:BOA} uses Squint with a discrete prior over a finite set of learning rates. It corresponds to BOA of \cite{Wintenberger2014}, where each expert is replicated multiple times with different constant learning rates. The proof (with the exact constants) is deferred to Appendix~\ref{app:proof_BOAfastrate}.

\begin{algorithm}[!t]
\caption{Squint -- BOA with multiple constant learning rates assigned to each parameter}
    \label{alg:BOA}
        {\bfseries Inputs:} $\Theta_0 = \{\theta_1,\dots,\theta_K\} \subset \cB_1$, $E>0$  and $\hat \pi_0 \in \Delta_K\footnotemark$.\\
    {\bfseries Initialization:} For $1 \leq i \leq \log(ET)$, define $\eta_{i} := (e^{i}E)^{-1}$ \\
    For $t=1,\dots,T$ 
    \begin{itemize}[label={--},topsep=2pt,parsep=2pt,itemsep=2pt]
      \item predict $\hat \theta_{t-1} = \sum_{k=1}^K \hat \pi_{k,t-1} \theta_k$ and observe $\nabla \ell_t(\hat \theta_{t-1})$,
      \item update component-wise for all $1\le k \le K$
        \[
            \hat \pi_{k,t} = \frac{\sum_{i=1}^{\log(ET)}\eta_{i}e^{\eta_{i} \sum_{s=1}^{t} (r_{k,s} - \eta_{i} r_{k,s}^2) }\pi_{k,0}}{\sum_{i'=1}^{\log(ET)}\E_{j\sim \hat \pi_0}\big[\eta_{i'}e^{\eta_{i'} \sum_{s=1}^{t} (r_{j,s} - \eta_{i'} r_{j,s}^2) }\big]} \,, 
            \ \text{where} \ r_{k,s} = \nabla \ell_t(\hat\theta_{s-1})^\top(\hat\theta_{s-1}-\theta_k) \,.
        \]
        .
    \end{itemize}
\end{algorithm}

\begin{theorem}
\label{thm:BOA_fastrate}
  Let $\Theta = \{\theta_1,\dots,\theta_K\} \subset\cB_1$ and $x>0$. Assume \ref{ass:a1} and \ref{ass:a2}. Apply Algorithm~\ref{alg:BOA} with grid $\Theta_0 = \Theta$, parameter $E = 4G/3$ and initial weight vector $\hat \pi_0 \in \Delta_K$\footnotetext{Throughout the paper, we denote the simplex of dimension $K\ge 1$ as $\Delta_K=\{\theta \in [0,\infty)^k;\; \|\theta\|_1=1, \|\theta\|_0=1\}$.}. Then, for all $T \geq 1$ and all $\pi \in \Delta_K$, with probability at least $1-2e^{-x}$
  \[
   \E_{k \sim \pi} \left[R_{T}(\theta_k)  \right]
      \lesssim  \left(\frac{\cK(\pi,\hat \pi_0)+ \log\log(GT)  + x }{\alpha T} \right)^{\frac{1}{2-\beta}} \,,
\]
where $\cK(\pi,\hat \pi_0) := \sum_{k=1}^K \pi_k \log(\pi_k/\hat \pi_{k,0})$ is the Kullback-Leibler divergence.
\end{theorem}

A fast rate of this type (without quantiles property) can be obtained in expectation by using the exponential weight algorithm (Hedge) for exp-concave loss functions. However, Theorem~\ref{thm:BOA_fastrate} is stronger. First, Assumption~\ref{ass:a2} only needs to hold on the risks $\E_{t-1}[\ell_t]$, which is much weaker than exp-concavity of the losses  $\ell_t$. It can hold  for absolute loss or quantile regression under regularity conditions. Second, the algorithm uses the so-called gradient trick. Therefore, simultaneously with upper-bounding the average risk $\smash{\cO(T^{-1/(2-\beta)})}$ with respect to the experts $(\theta_k)$, the algorithm achieves the slow rate $\smash{\cO(1/\sqrt{T})}$ with respect to any convex combination (similarly to EG). Finally, we recall that our result holds with high-probability, which is not the case for Hedge (see~\cite{Audibert2008}).

If the algorithm is run with a uniform prior $\hat \pi_0 = (1/K,\dots,1/K)$, Theorem~\ref{thm:BOA_fastrate} implies that for any subset $\Theta'\subseteq \Theta$, with high probability 
\[
  \textstyle{
    \max_{\theta \in \Theta'} R_T(\theta) \lesssim \left(\frac{\log(K/\Card(\Theta')) + \log\log (GT)}{\alpha T}\right)^{\frac{1}{2-\beta}} \,.}
\]
One only pays the proportion of good experts $\log(K/\Card(\Theta'))$ instead of the total number of experts $\log(K)$. This is the advantage of quantile bounds. We refer to~\cite{vanErvenKoolen2015} for more details, who obtained a similar result for the regret (not the average risk). Such quantile bounds on the risk were studied by \citet[Section~7]{mehta2016fast} in a batch i.i.d. setting (i.e., $\ell_t$ are i.i.d.). A standard online to batch conversion of our results shows that in this case, Theorem~\ref{thm:BOA_fastrate} yields with high probability for any $\pi \in \Delta_K$
\[ \textstyle{
  \E_T\Big[\ell_{T+1}(\bar \theta_T) - \E_{k\sim \pi}\big[\ell_{T+1}(\theta_k)\big]\Big] \lesssim  \left(\frac{\cK(\pi,\hat \pi_0)+ \log\log(GT)  + x }{\alpha T} \right)^{\frac{1}{2-\beta}} \quad \text{where} \quad \bar \theta_T = (\nicefrac{1}{T})\sum_{t=1}^T \hat \theta_{t-1} \,.}
\]
This improves the bound obtained by~\cite{mehta2016fast} who suffers the additional gap 
\[
  \textstyle{(e-1) \ \E_T\big[\E_{k\sim \pi}[\ell_{T+1}(\theta_k)]- \min_{\pi^*\in \Delta_K} \ell_{T+1}(\E_{j\sim\pi^*}[\theta_j])\big] \,.}
\]

\subsection{Prediction with expert advice}
\label{sec:experts}

The framework of prediction with expert advice is widely considered in the literature (see~\cite{Cesa-BianchiLugosi2006} for an overview). We recall now this setting and how it can be included in our framework. At the beginning of each round $t$, a finite set of $K \geq 1$ experts forms predictions $\smash{\b f_{t} = (f_{1,t},\dots,f_{K,t}) \in [0,1]^K}$ that are included into the history $\cF_{t-1}$. The learner then chooses a weight vector $\smash{\hat \theta_{t-1}}$  in the simplex $\smash{\Delta_K := \{\theta \in \R_+^K: \|\theta\|_1 = 1\}}$ and produces a prediction $\smash{\hat f_t := \hat \theta_{t-1}^\top \b f_t \in \R}$ as a linear combination of the experts. Its performance at time $t$ is evaluated thanks to a loss function\footnote{For instance, $g_t$ can be the square loss with respect to some observation $y \mapsto (y-y_t)^2$.} $g_t:\R \to \R$.  The goal of the learner is to approach the performance of the best expert on a long run. This can be done by minimizing the average risk
$
  \smash{R_{k,T} := \frac{1}{T} \sum_{t=1}^T \E_{t-1}[g_t(\hat f_t)] - \E_{t-1}[g_t(f_{k,t})] \,,}
$
with respect to all experts $k\in \{1,\dots,K\}$.

This setting reduces to our framework with dimension $d=K$. Indeed, it suffices to choose the $K$-dimensional loss function $\smash{\ell_t: \theta \mapsto g_t(\theta^\top \b f_t)}$ and the canonical basis $\smash{\Theta := \{\theta\in \R_+^K: \|\theta\|_1=1,\|\theta\|_0=1\}}$ in $\R^K$ as the reference set. Denoting by $\theta_k$ the $k$-th element of the canonical basis, we see that $\theta_k^\top \b f_t = f_{k,t}$, so that $\ell_t(\theta_k) = g_t(f_{k,t})$. Therefore, $R_{k,T}$ matches our definition of $R_T(\theta_k)$ in Equation~\eqref{eq:regret} and we get under the assumptions of Theorem~\ref{thm:BOA_fastrate} a bound of order:
\[
  \textstyle{  \E_{k \sim \pi}\big[R_{k,T}\big] \lesssim  \Big(\frac{\cK(\pi,\hat \pi_0)+\log\log(GT) + x}{\alpha T}\Big)^{\frac{1}{2-\beta}} \,.}
\]

It is worth to point out that though the parameters $\theta_k$ of the reference set are constant, this method can be used to compare the player with arbitrary strategies $f_{k,t}$ that may evolve over time and depend on recent data. This is why we do not want to assume here that there is a single fixed expert $k^* \in \{1,\dots,K\}$ which is always the best, i.e., $\E_{t-1}[g_t(f_{k^*,t})] \leq \min_k \E_{t-1}[g_t(f_{k,t})]$. Hence, we cannot replace \ref{ass:a2} with the closely related Bernstein assumption (see Ass.~\eqref{eq:Bernstein} or \cite[Cond.~1]{koolen2016combining}).

In this setting, Assumption~\ref{ass:a2} can be reformulated on the one dimensional loss functions $g_t$ as follows: there exist $\alpha >0$ and $\beta \in [0,1]$ such that for all $t \geq 1$, for all $0 \le f_1,f_2 \le 1$,
\[
    \E_{t-1}[g_t(f_1) - g_t(f_2)] \leq \E_{t-1}\big[ g'_t(f_1) (f_1 - f_2)\big]
    - \E_{t-1} \bigg[\left(\alpha \big(g'_t(f_1)(f_1 - f_2)\big)^{2} \right)^{1/\beta} \bigg]\,,\quad a.s.
\]
It holds with $\alpha =\kappa/(2G ^2)$ for $\kappa$-strongly convex risk $\E_{t-1}[g_t]$.  For instance, the square loss $g_t = (\,\cdot-y_t\,)^2$ satisfies it with $\beta = 1$ and $\alpha = 1/8$.

\section{Online optimization in the unit \texorpdfstring{$\ell_1$}{L1}-ball}
\label{sec:B1}
The aim of this section is to extend the preceding results to the reference set $\Theta = \cB_1$ instead of finite $\Theta = \{\theta_1,\dots,\theta_K\}$.
A classical reduction from the expert advice setting to the $\ell_1$-ball is the so-called ``gradient-trick''. A direct analysis on BOA applied to the 2d corners of the $\ell_1$-ball  suffers  a slow rate $\cO(\nicefrac{1}{\sqrt{T}})$ on the average risk. 
The goal is to exhibit algorithms that go beyond $\cO(\nicefrac{1}{\sqrt{T}})$. In view of the fast rate in Theorem \ref{thm:BOA_fastrate} the program is clear; in order to accelerate BOA, one has to add in the grid of experts some points of the $\ell_1$-ball to the 2d corners. In Section \ref{sec:discretization} one investigate the cases of non adaptive grids that are optimal but yields unfeasible (NP) algorithm. In Section \ref{sec:fix} we introduce a pseudo-metric in order to bound the regret of grids consisting of the 2d corners and some arbitrary fixed points. From this crucial step, we then derive the form of the adaptive points we have to add to the 2d corners, in the adversarial case, Section \ref{sec:adv}, and in the i.i.d. case, Section \ref{sec:acc}.

\subsection{Warmup: fast rate by discretizing the space}
\label{sec:discretization}

As a warmup, we show how to use Theorem~\ref{thm:BOA_fastrate} in order to obtain fast rate on $R_T(\theta)$ for any $\theta \in \cB_1$. Basically, if the parameter $\theta$ could be included into the grid $\Theta_0$, Theorem~\ref{thm:BOA_fastrate} would turn into a bound on the regret $R_T(\theta)$ with respect to $\theta$. However, this is not possible as we do not know $\theta$ in advance. A solution consists in approaching $\cB_1$ with $\cB_1(\epsilon)$, a fixed finite $\epsilon$-covering in $\ell_1$-norm of minimal cardinal. In dimension $d$, it is known that $\smash{\Card (\cB_1(\epsilon))  \lesssim \big(1/{\epsilon}\big)^{d}}$. We obtain the following near optimal rate for the regret on $\cB_1$.

\begin{proposition}
\label{prop:discretization}
Let $x>0$ and $T\geq 1$. Under Assumptions of Theorem~\ref{thm:BOA_fastrate}, applying Algorithm~\ref{alg:BOA} with grid $\Theta_0 =  \cB_1(T^{-2})$ and uniform prior $\hat \pi_0$ over $\Delta_{\Card(\cB_1(T^{-2}))}$ satisfies for all $\theta \in \cB_1$
\begin{equation*}
    R_T(\theta) \lesssim  \Big(\frac{d \log T + \log\log(GT) + x }{\alpha T}\Big)^{\frac1{2-\beta}} + \frac{G  }{T^2} \,.
\end{equation*}
\end{proposition}

\begin{proof}
Let $\epsilon = \nicefrac{1}{T}^2$ and $\theta  \in  \cB_1$ and $\tilde \theta $ be its $\epsilon$-approximation in $\smash{\cB_1(\epsilon)}$.  The proof follows from Lipschitzness of the loss: $R_T(\theta) \leq R_T(\tilde \theta) + G\epsilon$; followed  by applying Theorem~\ref{thm:BOA_fastrate} on $R_T(\tilde \theta)$. 
\end{proof}

Following this method and inspired by the work of~\cite{RigolletTsybakov2011}, one can improve $d$ to $\|\theta\|_0 \log d$ by carefully choosing the prior $\hat \pi_0$; see Appendix~\ref{app:discretization} for details.
The obtained rate is optimal up to log-factors. However, the complexity of the discretization is prohibitive (of order $T^d$) and non realistic for practical purpose.

\subsection{Regret bound for arbitrary fixed discretization grid}
\label{sec:fix}\label{sec:slowrate}

Let $\Theta_0 \subset \cB_1$ of finite size. The aim of this Section is to study the regret of Algorithm~\ref{alg:BOA} with respect to any $\theta \in \cB_1$ when applied with the grid $\Theta_0$. Similarly to Proposition~\ref{prop:discretization}, the average risk may be bounded as 
\begin{equation}
  \label{eq:crude}
    \textstyle{
    R_T(\theta) \lesssim  \Big(\frac{\log\Card(\Theta_0)+\log \log T +x }{\alpha T}\Big)^{\frac1{2-\beta}} + G \|\theta'-\theta\|_1 \,,}
\end{equation}
for any $\theta'\in \Theta_0$. We say that a regret bound is \emph{accelerable} if it provides a fast rate except a term depending on the distance with the grid (i.e., the term in $\|\theta'-\theta\|_1$ in~\eqref{eq:crude}) which vanishes to zero. This property will be crucial in obtaining fast rates by enriching the grid $\Theta_0$. Hence the regret bound~\eqref{eq:crude} is not accelerable due to the second term that is constant. In order to find an accelerable regret bound, we introduce the notion  of \emph{averaging accelerability}, a pseudo-metric that replaces the $\ell_1$-norm in~\eqref{eq:crude}.  We define it now formally but we will give its intuition in the sketch of proof of Theorem~\ref{thm:slowrate}.
\begin{definition}[averaging accelerability] \label{def:distance}
For any $\theta,\theta'\in \cB_1$, we define 
\begin{equation*}
  \textstyle{D(\theta,\theta') := \min \big\{0\leq \pi \leq 1:   \|\theta - (1-\pi) \theta'\|_1 \leq \pi \big\}\,.}
\end{equation*}
\end{definition}
This averaging accelerability has several nice properties. In Appendix~\ref{app:pseudometric}, we provide a few concrete upper-bounds in terms of classical distances. For instance,  Lemma~\ref{lem:rlarge} provides the upper-bound
$\smash{D(\theta,\theta') \leq \|\theta-\theta'\|_1/(1-\|\theta'\|_1\wedge\|\theta\|_1)}$. 
We are now ready to state our regret bound, when Algorithm~\ref{alg:BOA} is applied with an arbitrary approximation grid~$\Theta_0$.

\begin{theorem}
\label{thm:slowrate}
Let $x>0$. Let $\Theta_0 \subset \cB_1$ of finite size such that $\{\theta:\|\theta\|_1=1,\|\theta\|_0=1\}\subseteq \Theta_0$. Let Assumption~\ref{ass:a1} and \ref{ass:a2} be satisfied. Then, Algorithm~\ref{alg:BOA} applied with $\Theta_0$, uniform weight vector $\hat \pi_0$ over the elements of $\Theta_0$ and $E = 8 G/3$, satisfies with probability $1-e^{-x}$,
\[
  R_T( \theta)
  \lesssim   \left(\frac{a}{\alpha T} \right)^{\frac{1}{2-\beta}} + G  D(\theta,\Theta_0) \sqrt{\frac{a}{T}}  + \frac{aG}{T} \,,
\]
for all $\theta \in \cB_1$, where
$
a = \log \Card(\Theta_0) + \log \log (GT)+ x$ and $D(\theta,\Theta_0):=\min_{\theta'\in \Theta_0}D(\theta,\theta')$.
\end{theorem}
\begin{proof}[Sketch of proof]
The complete proof can be found in Appendix~\ref{app:proof_slowrate} but we give here the high-level idea of the proof. 
Let $\theta$ be the unknown parameter the algorithm will be compared with. Let $\theta'\in \Theta_0$ a point in the grid $\Theta_0$ minimizing $D(\theta,\theta')$.  
Then one can  decompose $\theta = (1-\epsilon)\theta'+ \epsilon \theta''$ for a unique point $\|\theta''\|_1=1$ and $\epsilon := D(\theta,\theta')$. See Appendix~\ref{app:proof_slowrate} for details. In the analysis, the regret bound with respect to $\theta$ can be decomposed into two terms:
\begin{itemize}[nosep,topsep=-4pt,parsep=0pt,itemsep=0pt,label={--}]
  \item The first one quantifies the cost of picking $\theta' \in \Theta_0$, bounded using Theorem~\ref{thm:BOA_fastrate}; 
  \item The second one is the cost of learning $\theta''\in \cB_1$ rescaled by $\epsilon$. Using a classical slow-rate bound in $\cB_1$, it is of order $\cO(1/\sqrt{T})$. 
\end{itemize}
The average risk $\Reg(\theta)$ is thus of the order
\[
    \mystrut(2,15) \smash{  (1-\epsilon)\underbrace{\Reg(\theta')}_{\text{Thm~\ref{thm:BOA_fastrate}}} + \epsilon \hspace*{-12pt}\underbrace{\Reg(\theta'')}_{G\sqrt{\log(\Card\Theta_0))/T}}  \hspace*{-10pt} \lesssim \Big(\frac{\log \Card(\Theta_0)+\log \log (GT) +x }{\alpha T}\Big)^{\frac1{2-\beta}} + \epsilon G \sqrt{\frac{\log \Card(\Theta_0)}{T}}}\,.  \qedhere
\] 
\end{proof}

 Note that the bound of Theorem~\ref{thm:slowrate} is \emph{accelerable} as it vanishes to zero on the contrary to Inequality~\eqref{eq:crude}. Theorem~\ref{thm:slowrate} provides an upper-bound which may improve the rate $\cO(\nicefrac{1}{\sqrt{T}})$ if the distance $D(\theta,\Theta_0)$ is small enough. By using the properties of the averaging accelerability (see Lemma~\ref{lem:rlarge} in Appendix~\ref{app:pseudometric}), Theorem~\ref{thm:slowrate} provides some interesting properties of the rate in terms of $\ell_1$ distance. By including $0$ into our approximation grid $\Theta_0$, we get a an oracle-bound of order $\cO(\nicefrac{\|\theta\|_1}{\sqrt{T}})$ for any $\theta \in \cB_1$. 
Furthermore, it also yields for any $\|\theta\|_1 \leq 1- \gamma < 1$, a bound of order $\smash{R_T(\theta) \leq \cO\big(\|\theta-\theta_k\|_1 /(\gamma \sqrt{T})\big)}$ for all $\theta_k \in \Theta_0$. 

It is also interesting to notice that the bound on the gradient $G$ can be substituted with the averaged gradient 
observed by the algorithm. This allows to replace $G$ with the level of the noise in certain situations with vanishing gradients (see for instance Theorem~3 of \citet{GaillardWintenberger2017}).

\subsection{Fast-rate sparsity regret bound under adversarial data}\label{sec:adv}

In this section, we focus on the adversarial case where $\ell_t = \E_{t-1}[\ell_t]$ are $\mu$-strongly convex deterministic functions. In this case, Assumption~\ref{ass:a2} is satisfied with $\beta =1$ and $\alpha = \mu/(2G^2)$.
Our algorithm, called BOA+, is defined as follows. For $i\geq 0$, it predicts from time step $t_i=2^i$ to $t_{i+1}-1$, by restarting Algorithm~\ref{alg:BOA} with uniform prior, parameter $E = 4G/3$  and updated discretization grid $\Theta_0$ indexed by~$i$: 
\[
  \Theta^{(i)} = \{[\theta^*_{i}]_k,k=1,\dots,d\} \cup \{\theta:\|\theta\|_1 = 2,\|\theta\|_0=1\}\,,
\]
where $
  \theta^*_i  \in \argmin_{\theta \in \cB_{1}} \sum_{t=1}^{t_i-1} \ell_t(\theta) 
$
is the empirical risk minimizer (or the leader) until time $t_i-1$. The notation $[\,\cdot\,]_{k}$ denotes the hard-truncation with $k$ non-zero values. Remark that $\theta^*_i$ for $i=1,2,\dots,\log_2(T)$ can be efficiently computed approximatively as the solution of a strongly convex optimization problem.

\begin{theorem} \label{thm:adversarial}
Assume the losses are $\mu$-strongly convex on $\cB_2:=\{\theta\in \R^d:\|\theta\|_1\leq 2\}$ with gradients bounded by $G$ in $\ell_\infty$-norm. The average regret of BOA+ is upper-bounded for all $\theta \in \cB_{1}$ as:
\[
   R_T(\theta) \leq \tilde \cO \left( \min\left\{ G \sqrt{\frac{\log d}{T}},  \sqrt{\frac{\|\theta\|_0}{\mu}} \left(G \sqrt{\frac{\log d}{T}}\right)^{\frac{3}{2}} ,  \frac{ \sqrt{\|\theta\|_0 d} G^2 \log d  }{\mu T}  \right\} \right) \,.
\]
\end{theorem}
The proof is deferred to the appendix. It is worth to notice that the bound can be rewritten as follows:
\[
    R_T(\theta) \leq \tilde \cO \left( 
  \min\left\{G \sqrt{\frac{\log d}{T}}, \frac{\|\theta\|_0 G^2 \log d  }{\mu T}\right\} \min\left\{G \sqrt{\frac{\log d}{T}}, \frac{d G^2 \log d  }{\mu T}\right\} \right)^{1/2} \,.
\]
It provides an intermediate rate between known optimal rates without sparsity 
$\cO(\sqrt{\log d/T})$ and $\smash{\tilde \cO(d/T)}$ and known optimal rates with sparsity $\smash{\cO(\sqrt{\log d/T})}$ and $\smash{\tilde \cO(\|\theta\|_0/T)}$ but with non-efficient procedures only. If all $\theta^*_i$ are approximatively $d_0$-sparse it is possible to achieve a rate of order $\smash{\tilde \cO(d_0/ T)}$, for any $\|\theta\|_0\le d_0 $. This can be achieved in particular in the i.i.d. setting (see next section). However, we leave for future work whether it is possible to achieve it in full generality and efficiently in the adversarial setting.

\begin{remark} The strongly convex assumption on the losses can be relaxed by only assuming Inequality~\eqref{eq:relaxedstrongconvexity}: it exists  $\mu > 0$ and $\beta \in [0,1]$ such that for all $t\geq 1$ and $\theta \in \cB_{1}$ 
\begin{equation}\label{eq:assadv}
  \textstyle{ \mu \big\|\theta - \theta^*_{t}\big\|_2^2  \leq  \Big( \frac{1}{t} \sum_{s=1}^{t} \ell_s(\theta) - \ell_s(\theta^*_t)\Big)^\beta}, \quad \text{where} \quad  \theta^*_t \in  \argmin_{\theta \in \cB_{1}} \sum_{s=1}^{t} \ell_s(\theta)\,.
\end{equation}
 The rates will then depend on $\beta$ as it was the case in Theorem~\ref{thm:BOA_fastrate}.  A specific interesting case is when  $\|\theta_t^*\|_1=1$. Then $\theta_t^*$ is very likely to be sparse. Denote $S_t^*$ its support. Assumption \eqref{eq:assadv} can be weakened in this case. Indeed any $\theta\in \cB_1$ satisfies $\|\theta\|_1\le \|\theta_t^*\|_1$, which from Lemma 6 of \cite{AgarwalNegahbanWainwright2012} yields $\|\theta-\theta_t^*\|_1\le 2\|[\theta-\theta_t^*]_{S^*_t}\|_1 $ where $[\theta]_{S}=(\theta_i\1_{i\in S})_{1\le i\le d}$. One can thus restrict Assumption \eqref{eq:assadv} to hold on the support of $\theta^*_t$ only.  Such restricted conditions for $\beta=1$ are common in the sparse learning literature and essentially necessary to hold for the existence of efficient and optimal sparse procedures, see \cite{Zhang2014}. In the online setting, the restricted condition \eqref{eq:assadv} with $\beta=1$ should hold at any time $t\ge1$, which is unlikely. 
\end{remark}

\subsection{Fast-rate sparsity risk bound under i.i.d. data}
\label{sec:acc}

In this section, we provide an algorithm with fast-rate sparsity risk-bound on $\cB_1$ under i.i.d. data. This is obtained by regularly restarting Algorithm~\ref{alg:BOA} with an updated discretization grid $\Theta_0$ approaching the set of minimizers 
$
 \Theta^* := \arg\min_{ \theta \in \cB_1}\E[\ell_t(\theta)]
$.

In this setting, a close inspection of the proof of Theorem~\ref{thm:acceleratedBOA} shows that we can replace Assumption~\ref{ass:a2} with the Bernstein condition: it exists $\alpha'>0$ and $\beta \in [0,1]$, such that for all $\theta \in \cB_1$, all $\theta^* \in \Theta^*$ and all $t\geq 1$, 
  \begin{equation}
    \label{eq:Bernstein}\tag{A2'}
    \qquad \alpha' \E \Big[ \big(\nabla \ell_t(\theta)^\top(\theta - \theta^{*})\big)^{2}  \Big] \leq \E\Big[ \nabla \ell_t(\theta)^\top (\theta - \theta^*)\Big]^{\beta}\,.
  \end{equation}
This fast-rate type stochastic condition is equivalent to the \emph{central condition} (see \cite[Condition 5.2]{van2015fast}) and was already considered to obtain faster rates of convergence for the regret (see \cite[Condition~1]{koolen2016combining}).

\paragraph{The Łojasiewicz's assumption}

In order to obtain sparse oracle inequalities we also need the Łojasiewicz's Assumption~\ref{ass:Lojasiewicz} which is a relaxed version of strong convexity of the risk.
\begin{enumerate}[label={(A\arabic*)},topsep=0pt]
  \setcounter{enumi}{\value{saveenum}}
  \item\label{ass:Lojasiewicz} \emph{Łojasiewicz's inequality:} $(\ell_t)_{t\geq 1}$ is i.i.d. and it exists $\beta \in [0,1]$ and $0< \mu\le 1$ such that, for all $\theta \in \R^d$ with $\|\theta\|_1\leq 1$ , it exists $\theta^* \in \Theta^* \subseteq \cB_1$ satisfying
  \[
    \mu  \big\| \theta -  \theta^\ast\big\|_2^2 \leq \E[\ell_t( \theta) - \ell_t( {\theta^\ast})]^\beta\,.
  \]
    \setcounter{saveenum}{\value{enumi}}
\end{enumerate}

This assumption is fairly mild. It is indeed satisfied with $\beta = 0$  and $\mu = 1$ as soon as the loss is convex. For $ \beta = 1$, this assumption is implied by the strong convexity of the risk $\E[\ell_t]$.   
One should mention that our framework is more general than this classical case because \vspace{-5pt}
\begin{itemize}[label={-},itemsep=0pt,parsep=0pt,topsep=0pt,leftmargin=15pt]
\item multiple optima are allowed, which seems to be new when combined with sparsity bounds;
 \item on the contrary to~\cite{Steinhardt2014} or~\cite{GaillardWintenberger2017}, our framework does not compete with the minimizer $\theta^*$ over $\R^d$ with a known upper-bound on the $\ell_1$-norm $\|\theta^*\|_1$. We consider the minimizer over the $\ell_1$-ball $\cB_1$ only. The latter is more likely to be sparse and Assumption~\ref{ass:Lojasiewicz} only needs to hold over $\cB_1$. 
\end{itemize}

Assumption~\ref{ass:a2} (or~\eqref{eq:Bernstein}) and~\ref{ass:Lojasiewicz} are strongly related. Assumption~\ref{ass:Lojasiewicz} is more restrictive because it is heavily design dependent. In linear regression for instance, the constant $\mu$ corresponds to the smallest non-zero eigenvalue of the Gram matrix while $\alpha = 1/G^2$.
If $\Theta^* = \{\theta^*\}$ is a singleton than Assumption~\ref{ass:Lojasiewicz} implies Assumption~\eqref{eq:Bernstein} with $\alpha' \geq \mu/G^2$. 

\paragraph{Algorithm and risk bound}

 Our new procedure is described in Algorithm~\ref{alg:SABOA}. It is based on the following fact: the bound of Theorem~\ref{thm:slowrate} is small if one of the estimators in $\Theta_0$ is close to $\Theta^\ast$. Thus, our algorithm regularly restarts BOA by adding current estimators of $\Theta^*$ into an updated grid $\Theta_0$. The estimators are built by averaging past iterates $\smash{\hat \theta_{t-1}}$ and truncated to be sparse and ensure small $\ell_1$-distance.  Remark that restart schemes under Łojasiewicz's Assumption is natural and was already used for instance in optimization by~\cite{roulet2017sharpness}.  
 A stochastic version of the algorithm (sampling randomly a subset of gradient coordinates at each time step) can be implemented as in the experiments of \cite{DuchiEtAl2008}.  We get the following upper-bound on the average risk. The proof, that computes the exact constants, is postponed to Appendix~\ref{app:proof_SABOA}.

 \begin{algorithm}[!t]
    \caption{SABOA -- Sparse Acceleration of BOA}
    \label{alg:SABOA}

    {\bfseries Parameters:} $E >0$ \\
    {\bfseries Initialization:} $t_i = 2^{i}$ for $i\geq 0$, \\
    For each  $i = 0,\dots$
    \begin{itemize}[itemsep=1pt,parsep=1pt,topsep=1pt]
      \item define $\smash{\bar  \theta^{(i-1)} := 0}$ if $i=0$ and $\smash{\bar  \theta^{(i-1)} :=  2^{-i+1} \sum_{t=t_{i-1}}^{t_i - 1} \hat  \theta_{t-1}}$ otherwise.
      \item  Define $\Theta^{(i)}$ a set of hard-truncated and dilated soft-thresholded versions of $\bar \theta^{(i-1)}$ as in~\eqref{eq:defgrid};
			\item Denote $\smash{K_i := \mathrm{Card}(\Theta^{(i)}) + 2d \leq (i+1)(1+\log d)+3d}$\,;
      \item At time step $t_i$, restart Algorithm~\ref{alg:BOA} in $\Delta_{K_i}$ with parameters $\Theta_0 := \Theta^{(i)} \cup \{\theta:\|\theta\|_1=1,\|\theta\|_0=1\}$ (denote by $\theta_1,\dots,\theta_{K_i}$ its elements), $E>0$ and uniform prior $\hat \pi_0$ over $\Delta_{K_i}$. In other words, for time steps $t= t_i,\dots,t_{i+1}-1$:
        \begin{itemize}[nosep,itemsep=0pt,parsep=0pt,topsep=0pt]
          \item predict $\smash{\hat  \theta_{t-1} = \sum_{k=1}^{K_i} \hat \pi_{k,t-1} \theta_k}$ and observe $\smash{\nabla \ell_t(\hat \theta_{t-1})}$
          \item define component-wise for all $1\leq k \leq K_i$
          \[{}
              \hat \pi_{k,t} = \frac{\sum_{i=1}^{\log(ET^2)}\eta_{k,i}e^{\eta_{k,i} \sum_{s=1}^{t} (r_{k,s} - \eta_{k,i} r_{k,s}^2) }\pi_{k,0}}{\sum_{i=1}^{\log(ET^2)}\E_{\pi_0}\big[\eta_{j,i}e^{\eta_{j,i} \sum_{s=1}^{t} (r_{j,s} - \eta_{j,i} r_{j,s}^2) }\big]} \,,
        \]
        where $r_{k,s} = \nabla \ell_t(\hat\theta_{s-1})^\top(\hat\theta_{s-1}-\theta_k) $.
        \end{itemize}
    \end{itemize}
\end{algorithm}

    \begin{theorem}
        \label{thm:acceleratedBOA}
        Let $x >0$, $\gamma \geq 0$. Under Assumptions~(A1-3), if $\Theta^* \subseteq \cB_{1-\gamma}$, $E = 4/3G \geq 1$, Algorithm~\ref{alg:SABOA} satisfies with probability at least $1-e^{-x}$ the bound on the average risk
        \[
        R_T( \theta^\ast)  \lesssim  \left(\frac{\log d+\log \log (GT) + x}{T}  \left(\frac{1}{\alpha}+ \frac{G^2}{\mu}\Big(d_0^2 \wedge \frac{d_0}{\gamma^2}\Big) \right)\right)^{\frac{1}{2-\beta}} \,,
        \]
      where $d_0 = \max_{\theta^*\in\Theta^*}\|\theta^*\|_0$. 
    \end{theorem}

    Let us conclude with some important remarks about Theorem~\ref{thm:acceleratedBOA}. First, it is worth pointing out that SABOA does not need to know the parameters $\delta$, $\beta$, $\alpha$, $\mu$ and $d_0$ to fulfill the rate of Theorem~\ref{thm:acceleratedBOA}.

    \emph{Approximately sparse optima}.  Our results can be extended to a unique approximately sparse optimum $\theta^*$. We get 
    $R_T(\theta) \leq (1+o(1))\|\theta - \theta^*\|_1 + \tilde \cO((\|\theta\|_0^2/T)^{1/(2-\beta)})$ for any $\theta \in \cB_1$; see~\cite{AgarwalNegahbanWainwright2012,bunea2007}. 

    \emph{On the radius of L1 ball}. We only performed the analysis into $\cB_1$, the $\ell_1$-ball of radius 1. However, one might need to compare with parameters into $\cB_1(U)$ the $\ell_1$-ball of radius $U>0$. This can be done by simply rescaling the losses and applying our results to the loss functions $\theta \in \cB_1 \mapsto \ell_t(U\theta)$ instead of $\ell_t$. 
    If $\theta^*$ lies on the border of the $\ell_1$-ball, we could not avoid a factor $\|\theta^*\|_0^2$. In that situation, our algorithm needs to recover the support of $\theta^*$ without the Irreprensatibility Condition~\citep{wainwright2009sharp} (see configuration 3 of Figure~\ref{fig:geom}). In this case, we can actually relax Assumption~\ref{ass:Lojasiewicz} to hold in sup-norm. 
    
\section*{Conclusion} In this paper, we show that BOA is an optimal online algorithm for aggregating predictors under very weak conditions on the loss. Then we aggregate sparse versions of the leader (BOA+) or of the averaging of BOA's weights (SABOA) in the adversarial or in the i.i.d. setting, respectively. Aggregating both achieves sparse fast-rates of convergence in any case. These rates are deteriorated compared with the optimal one that require restrictive assumption. Our  weaker conditions are very sensitive to the radius of the $\ell_1$-ball we consider. The optimal choice of the radius, if it is not imposed by the application, is left for future research.

\setcitestyle{square,numbers,sort&compress}
\bibliographystyle{abbrvnat}
\bibliography{biblio}

\cleardoublepage

\appendix

\section{Sparse oracle inequality by discretizing the space}
\label{app:discretization}

Inspired by the work of~\cite{RigolletTsybakov2011}, one can improve $d$ to $\|\theta\|_0 \log d$ in Proposition~\ref{prop:discretization} by carefully choosing the prior $\hat \pi_0$. To do so, we cover $\cB_1$ by the subspaces
\[
  \cB_1^\tau := \Big\{ \theta \in   \cB_1: \forall i \quad \tau_i = 0 \Rightarrow  \theta_i = 0 \Big\}\,,
\]
where $\tau \in \{0,1\}^d$ denotes a sparsity pattern which determines the non-zero components of $ \theta \in \cB_1^\tau$. For each sparsity pattern $\tau \in\{0,1\}^d$, the subspace $\cB_1^\tau$ can be approximated in $\ell_1$-norm by an $\epsilon$-cover $\smash{\cB_1^\tau(\epsilon)}$ of size $\smash{\epsilon^{-\|\tau\|_0}}$.
 In order to obtain the optimal rate of convergence, we apply Algorithm~\ref{alg:BOA} with $\smash{\Theta_0 = \cup_{\tau \in \{0,1\}^d} \cB_1^\tau(\epsilon)}$ with a non-uniform prior $\hat \pi_0$. The latter penalizes non-sparse   $\tau$ to reflect their respective complexities. We assign to any  $\theta \in \cB_1^\tau(\epsilon)$ the prior, depending on $\tau\in \{0,1\}^d$, 
\[
  \hat \pi_{\tau,0} = \left(\#\cB_1^\tau(\epsilon)(d+1) \binom{d}{d_0}\right)^{-1} \approx \frac{\epsilon^{d_0}}{(d+1) \binom{d}{d_0}} \qquad \text{where} \quad  d_0 = \| \tau\|_0\,.
\]
Note that the sum $\hat \pi_{\tau,0}$ over $\theta \in \cB_1^\tau(\epsilon)$ and $\tau \in \{0,1\}^d$ is one. Therefore, Theorem~\ref{thm:BOA_fastrate} yields
\begin{equation}
    R_T( \theta)
      \lesssim  \left(\frac{\|\theta\|_0 \log(d T/\|\theta\|_0) + x }{\alpha T}\right)^{\frac1{2-\beta}} + \frac{\|\theta\|_0 G }{T^2} \,,
    \label{eq:fastrate_discretized0}
\end{equation}
by noting that $\binom{d}{\|\theta\|_0} \leq d^{\|\theta\|_0}$ and choosing $\epsilon = \|\theta\|_0 /T^2$.  Similar optimal oracle inequalities for mixing arbitrary regressions functions are obtained by \citet{Yang2004} and \citet{Catoni1999}.

\section{Properties of the averaging accelerability}
\label{app:pseudometric}
In this appendix, we give a geometric interpretation of the \emph{averaging accelerability} defined in Definition~\eqref{def:distance}. We also provide  several properties in terms of classical distances.

\paragraph{Geometric insight} Let $\theta \in \cB_1$ be some unknown parameter and $\theta'\in \cB_1$ a point approximating $\theta$.
Let us define $\theta'' \in \cB_1$ the unique point satisfying 
\begin{equation}
  \label{eq:extrapolation}
  \|\theta'' \|_1=1 \qquad \text{and} \qquad \theta''=\lambda(\theta - \theta') + \theta'
\end{equation}
for some $\lambda \ge 1$. From this definition, we immediately derive that 
\[
  \left\|\theta - \Big(1-\frac{1}{\lambda}\Big)\theta'\right\|_1 = \frac{\|\theta''\|_1}{\lambda} = \frac{1}{\lambda}
\]
Therefore from Definition~\ref{def:distance}, we have $D(\theta,\theta')\leq \frac{1}{\lambda}$. Actually, this is an equality and we can write
\[
 D(\theta,\theta') = \max \Big\{\lambda \ge 1:   \|\lambda(\theta - \theta') + \theta '\|_1 \leq 1 \Big\}^{-1}\,.
\]
As the maximum is achieved, the averaging accelerability corresponds to the inverse of $\lambda$ in the definition \eqref{eq:extrapolation} of the extrapolation point $\theta''$.

\begin{figure}[h!]
   \begin{tikzpicture}[scale=2.2]
    \tikzset{cross/.style={cross out, draw=black, minimum size=2*(#1-\pgflinewidth), inner sep=0pt, outer sep=0pt},
cross/.default={2pt}}
        
\draw[fill=green,draw=black,fill opacity=.3] (-1,0) --(0,1) -- (1,0) ;
\draw[fill] (.2,.7) circle [radius=0.025];
\node [below] at (.2,.7) {$\theta$};
\draw[fill] (0.05,.65) circle [radius=0.025];
\node [below] at (0.05,.65) {$\theta'$};
\draw[fill] (0.275,.725) circle [radius=0.025];
\node [above] at (0.275,.725) {$\theta''$};
\draw[dotted,thick] (0.05,.65) --(0.275,.725);

\draw[fill=red,draw=black,fill opacity=.3] (1,0) --(2,1) -- (3,0) ;
\draw[fill] (2.05,.65) circle [radius=0.025];
\node [below] at (2.05,.65) {$\theta'$};
\draw[fill] (2.275,.725) circle [radius=0.025];
\node [above] at (2.5,.725) {$\theta=\theta''$};

\draw[fill=green,draw=black,fill opacity=.3] (3,0) --(4,1) -- (5,0) ;
\draw[fill] (4.05,0.95) circle [radius=0.025];
\node [below] at (4.05,.95) {$\theta'$};
\draw[fill] (4.275,.725) circle [radius=0.025];
\node [above] at (4.275,.725) {$\theta$};
\draw[fill] (5,0) circle [radius=0.025];
\node [above] at (5,0) {$\theta''$};

    \end{tikzpicture}
   \caption{Averaging accelerability for 3 different configurations.}\label{fig:geom}
\end{figure}
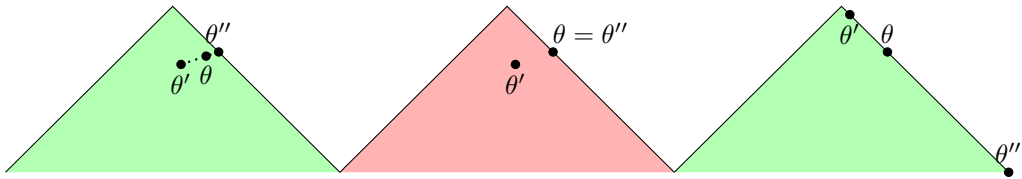

Figure \ref{fig:geom} pictures several configurations of $\theta'$ and $\theta$ that lead to different averaging accelerability. The further $\theta''$ is from $\theta$, the smaller is $D(\theta,\theta')$ and the smaller is the averaging accelerability. 
When $D(\theta,\theta') = 1/\lambda = 1$, then $\theta=\theta''$ and our regret bound does not improve the classic slow-rate $\cO(1/\sqrt{T})$. That typically happens when $\|\theta\|_1=1$, as in the second configuration in Figure~\ref{fig:geom}. In this case, a possible solution is to consider a larger ball (for instance of radius 2 instead of 1). This approach was considered in \cite{GaillardWintenberger2017}, see Figure~\ref{fig:geom} there. 
Another solution is to remark that even when $\|\theta\|_1=1$, the procedure is still accelerable ($D(\theta,\theta')<1$) if the approximation $\theta'$ satisfies the same constraints than $\theta$ (see the third configuration in Figure~\ref{fig:geom} where $\theta''$ and $\theta$ are on the same edge of the ball). We make this statement more precise in the following subsections.

\subsection{The averaging accelerability in terms of classical distances}
We provide in the next Lemmas a few  concrete upper-bounds in terms of classical distances. The proofs are respectively postponed to Appendices~\ref{app:rlarge} to \ref{app:softthreshold}. The first Lemma, states that the averaging accelerability  can be upper-bounded with the $\ell_1$-distance. 

\begin{lemma}
  \label{lem:rlarge}
  We have for any $\theta,\theta' \in \cB_1$
    $$
        D(\theta,\theta') \leq \frac{\|\theta - \theta'\|_1}{\|\theta - \theta'\|_1 + 1 -\|\theta\|_1}  \,.
    $$
\end{lemma}
The Lemma above has a main drawback. The averaging accelerability  does not decrease with the $\ell_1$-distance if $\|\theta\|_1 = 1$. 
In this case, we thus need additional assumptions. The following Corollary upper-bounds the averaging accelerability in sup-norm as soon as a $\theta'$ has a support included into the one of $\theta$. This situation is represented in the third configuration of Figure~\ref{fig:geom}.

\begin{lemma}
  \label{lem:supportmin}
  Let $\theta,\theta' \in \cB_1$. Assume that $\|\theta'\|_1\geq \|\theta\|_1$ and $\sign(\theta_i') \in \{0,\sign(\theta_i)\}$ for all $1\leq i\leq d$.
  Then,
  \[
    D(\theta,\theta') \leq 1- \min_{1\leq i\leq d} \frac{|\theta_i|}{|\theta'_i|}  \leq  \frac{\|\theta - \theta'\|_\infty}{\Delta} \,,
  \]
  where $\Delta := \min_{i:\theta'_i \neq 0} |\theta_i|$.
\end{lemma}

We want to emphasis here the two very different behavior of the averaging accelerability; 
\begin{itemize}[label={--},itemsep=0pt,parsep=0pt,topsep=0pt]
\item in the case $\|\theta\|_1< 1$: the averaging accelerability is proportional to $\|\theta-\theta'\|_1$.
\item in the case $\|\theta\|_1=1$: the averaging accelerability may be smaller than 1 and lead to improved regret guarantees under extra assumptions: $\|\theta'\|_1=1$ and the support of $\theta'$ is included in the one of $\theta$. The relative gain is then proportional to $\|\theta\|_0\|\theta-\theta'\|_\infty$.
\end{itemize}

\subsection{The averaging accelerability with an approximation in sup-norm in hand}

Let us focus on the second case, where the averaging accelerability is controlled under the knowledge of the support of $\theta$. 
The second inequality in Lemma \ref{lem:supportmin} is interesting but yields an undesirable dependence on $\Delta := \min_{i:\theta_i \neq 0} |\theta_i|$, which can be arbitrarily small and which is at best of order $\|\theta\|_1/\|\theta\|_0$.  Moreover, the recovery of the support of $\theta$ is a well studied difficult problem, see \cite{wainwright2009sharp}. Thanks to the following Lemma, we ensure the averaging accelerability from any $\ell_{\infty}$-approximation $\theta'$ of $\theta$. We use a dilated soft-thresholding version of $\theta'$ as an approximation of $\theta$. For any $\epsilon>0$, let us introduce $S_\epsilon$ the soft threshold operator so that $S_\epsilon(x)_i=\sign(x_i)(|x_i|-\epsilon)_+$ for all $1\le i\le d$. The soft threshold operator is  equivalent to the popular LASSO algorithm in the orthogonal design setting for the square loss. We couple the soft-thresholding with a dilatation that has the benefit of ensuring non thresholded coordinates faraway from zero. This allows to get rid of the unwanted factor $1/\Delta$ of the Lemma \ref{lem:supportmin}. It is replaced with a factor $2\|\theta\|_0/\|\theta\|_1$ which corresponds (up to the factor 2) to the best possible scenario for the value of $\Delta$.

\begin{lemma}
  \label{lem:softthreshold}
Let $\theta, \theta' \in \cB_1$ such that $\|\theta- \theta'\|_\infty \leq \epsilon$ and $\|\theta\|_0\leq d_0$. Then, define the dilated soft-threshold 
  \[
    \tilde \theta := S_\epsilon(\theta') \left(1+\frac{2d_0\epsilon}{\|S_\epsilon(\theta')\|_1}\right)\wedge \frac{1}{\|S_\epsilon(\theta')\|_1}
  \]
where by convention $\tilde \theta=0$ when $S_\epsilon(\theta')=0$. Then $\tilde \theta$ satisfies 
\begin{itemize}[label={--},topsep=-4pt,itemsep=0pt,parsep=0pt]
  \item[(i)] $\|\tilde \theta\|_1 \geq \|\theta\|_1$ if $\tilde \theta \neq 0$
  \item [(ii)] $\sign(\tilde\theta_i) \in \{0,\sign( \theta_i)\}$ for all $1\leq i\leq d$
  \item[(iii)] $
  D(\theta,\tilde \theta) \leq 2d_0\epsilon / \|\theta\|_1 \,.
$
\end{itemize}
\end{lemma}
Performing this transformation requires the knowledge of the values of $\epsilon$ and $d_0$ that are not observed. However, performing an exponential grid on $\epsilon$ from $\nicefrac{1}{T}$ to $U$ only harms the complexity by a factor $\log(UT)$.

\section{Proofs}

\subsection{Proof of Theorem~\ref{thm:BOA_fastrate}}
\label{app:proof_BOAfastrate}
  
  Algorithm~\ref{alg:BOA} is a particular case of the Bernstein Online Aggregation algorithm (BOA) with fixed learning rates of~\cite{Wintenberger2014}\footnote{It is also a specific case of Squint of~\cite{vanErvenKoolen2015} with a discrete distribution over the learning rates} applied on a particular set of experts $\cK$. We make more clear the connexion thereafter.  We start our proof with Theorem~3.2 of \cite{Wintenberger2014} that states that for any distribution $\tilde \pi$ over the set of experts $j\in \cK$:
    \begin{equation}
      \sum_{t=1}^T \E_{j\sim \tilde \pi}[r_{j,t}] \leq  \E_{j \sim \tilde \pi} \left[\eta_j \sum_{t=1}^T r_{j,t}^2  + \frac{\log (\tilde \pi_j/ \tilde \pi_{j,0})}{\eta_j}\right]\,,
      \label{eq:regretBOAmeta}
    \end{equation}
  where $r_{j,t} = \nabla \ell_t(\hat \theta_{t-1})^\top (\hat \theta_{t-1} - \theta_{j})$, where $\tilde \pi_{j,0}$ are the initial weights assigned to the experts by the algorithm. In the original version of the BOA algorithm, each expert $\theta_k$ is assigned to a single learning rate $\eta_k$. In Algorithm~\ref{alg:BOA} each parameter $\theta_k$ for $k=1,\dots,K$ is replicated several times, each replicate being assigned a different learning rate $\eta_i = e^{-i} E^{-1}$ for $1\le i \le \log (ET^2)$. Algorithm~\ref{alg:BOA} corresponds  to applying BOA on experts indexed by couples $j = (k,i)$ of a parameter $\theta_k$ for $k=1,\dots,K$ and a learning-rate $\eta_i = e^{-i} E^{-1}$. Each couple $j=(k,i)$ is assigned the initial weight $\tilde \pi_{j,0} = \hat \pi_{k,0}/\log(ET^2)$.  We  refer to these couples of parameter-learning rate $j=(k,i)$ as experts. 
  
 For each parameter $\theta_k,k \in \{1,\dots,K\}$, let $1 \leq i_k \leq \log (ET^2)$ be the index of a learning rate which will be chosen later by the analysis in order to optimize the final bound. Let $\pi$ be a distribution over the index set $\{1,\dots,K\}$. We now apply Inequality~\eqref{eq:regretBOAmeta} to a specific distribution $\tilde \pi$ on the experts. We choose $\tilde \pi$ so that it assigns all the mass $\pi_k$ on the expert $(k,i_k)$ and no mass on the experts $(k,i)$ for $i\neq i_k$. In other words, $\tilde \pi_j = \pi_k \indic_{i=i_k}$. Then $\log(\tilde \pi_j/ \tilde \pi_{j,0})=\log(\pi_k/\hat \pi_{k,0} \log (ET^2))$ and Inequality~\eqref{eq:regretBOAmeta} entails
\begin{align}
\sum_{t=1}^T \E_{k \sim \pi}[r_{k,t}]
	&  \leq  \E_{k \sim \pi}\left[\underbrace{e^{-i_k }E^{-1}}_{:=\lambda_k}\sum_{t=1}^T    r_{k,t}^2 + e^{i_k} E \big(\log (\pi_k/ \hat \pi_{k,0})+\log\log (ET^2)\big)\right] \nonumber \\
	& =  \E_{k \sim \pi}\left[ \lambda_k \sum_{t=1}^T   r_{k,t}^2 + \frac{ \log (\pi_k/ \hat \pi_{k,0})+\log\log (ET^2)}{\lambda_k }\right]\,, \label{eq:regretlambdak0}
\end{align}
where we defined $\lambda_k := e^{-i_k}E^{-1}$. Now, by choosing $i_k$, this bound may be optimized with respect to any $\lambda_k$ of the form $e^{-i_k}E^{-1}$, with $1 \leq i_k \leq \log (ET^2)$. To get the minimum over any $\lambda_k >0$, we pay additional additive and multiplicative terms due to edge effects that we compute now.
Fix $k>0$ and define $\smash{V_k = \sum_{t=1}^T r_{k,t}^2}$. The minimum is reached when both terms in~\eqref{eq:regretlambdak0} are equal. This yields the optimal choice $\lambda _k \approx (V_k/a)^{-1/2}$, where $a_k := \log(\pi_k/\pi_{k,0}) + \log \log (ET^2)$. However, because of edge effects, this is only possible when $1/(ET)^2 \leq (V_k/a)^{-1/2} \leq 1/(Ee)$. We distinguish three cases:
   \begin{itemize}
      \item if $\sqrt{{a_k}/{V_k}} > 1/(eE)$: then, we choose $\lambda_k = 1/(eE)$, which yields:
      \[
          \lambda_k V_k + \frac{a_k}{\lambda_k} \leq \frac{2a_k}{\lambda_k} = 2ea_kE \leq 6a_kE
      \]
      \item if $1/(ET)^2 \leq (V_k/a_k)^{-1/2} \leq 1/(Ee)$: then, we can choose $\lambda_k$ such that
      \[
          \frac{\lambda_k}{\sqrt{e}} \leq  (V_k/a_k)^{-1/2} \leq \sqrt{e} \lambda_k \,,
      \]
      which entails $\lambda_k V_k + \frac{a_k}{\lambda_k} \leq  2\sqrt{e} \sqrt{a_k V_k}  \leq 4 \sqrt{a_k V_k}$
      \item if $\sqrt{{a_k}/{V_k}} < (ET)^{-2}$: then, the choice $\lambda_k  = (ET)^{-2}$ gives
      \[
          \lambda_k V_k + \frac{a_k}{\lambda_k} \leq 2 \lambda_k V_k  = \frac{2V_k}{E^2 T^2} \leq  \frac{2}{T} \,,
      \]
      because $r_{k,t}^2 \leq E^2$.
   \end{itemize}
   Putting the three cases together and plugging into Inequality~\eqref{eq:regretlambdak0} yields
  \begin{equation}
    \label{eq:firstregret}
     \sum_{t=1}^T \E_{k \sim \pi}[r_{k,t}] \leq  \E_{k \sim \pi}\left[ 4 \sqrt{a_k V_k}  + 6 a_k E\right] + \frac{2}{T}.
  \end{equation}
   We recall Young's inequality.
   \begin{lemma}[Young's inequality]
  \label{lem:young}
  For all $a,b\geq 0$ and $p,q >0$ such that $1/p+1/q = 1$, then
  $ab \leq a^p/p + b^q/q$.
\end{lemma}
 Applying it, with $p=q=2$, and $a = \sqrt{2\lambda_kV_k}$ and $b=\sqrt{8a_k/\lambda_k}$, we get $4\sqrt{a_kV_k} \leq \lambda_k V_k + 4a_k/\lambda_k$ for any $\lambda_k >0$. Therefore, substituting into Inequality~\eqref{eq:firstregret}, for any distribution $\pi$ over $\{1,\dots,K\}$, we have
   \begin{equation}
      \sum_{t=1}^T \E_{k \sim \pi}[r_{k,t}] \leq  \E_{k \sim \pi}\left[  \lambda_k V_k  +  \frac{4a_k}{\lambda_k} + 6 a_k E\right] + \frac{2}{T}\,,
      \label{eq:regretlambdak}
   \end{equation}
  where we recall that $V_k = \sum_{t=1}^T r_{k,t}^2$ and $a_k = \log(\pi_k/\pi_{k,0}) + \log \log (ET^2)$. For simplicity, from now on, we will denote $\E_{k \sim \pi}$ by $\E_{\pi}$. Using Theorem 4.1  of \cite{Wintenberger2014} for $\eta_{j,t}=\lambda_j$ independent of $t$, we obtain with probability $1-e^{-x}$ and integrating with respect to $\pi$
\begin{align}
\sum_{t=1}^T \E_{t-1}[\E_\pi[r_{k,t}]]&\le \sum_{t=1}^T \E_\pi[r_{k,t}] +\E_\pi\left[ \lambda_k \sum_{t=1}^T   r_{k,t}^2 + \frac{x}{\lambda_k }\right] \nonumber \\
& \stackrel{\eqref{eq:regretlambdak}}{\le}  \E_\pi\left[ 2\lambda_k  \sum_{t=1}^T   r_{k,t}^2 + \frac{x+4a_k}{\lambda_k} + 6 a_k E\right] + \frac{2}{T} \label{eq:risklambdak}.
\end{align}
To apply Assumption~\ref{ass:a2}, we need to transform the second order term (the sum of $r_{k,s}^2$ in the right-hand side) into a cumulative risk. This can be done using a Poissonian inequality for martingales (see for instance Theorem 9 of \cite{GaillardWintenberger2017}): with probability at least $1-e^{-x}$
\[
  \sum_{t=1}^T r_{k,t}^2 \leq 2 \sum_{t=1}^T \E_{t-1}\big[r_{k,t}^2] + \frac{9}{4} E^2 x \,.
\]
Substituting into the previous regret inequality, this yields for any $\lambda_k>0$ and any distribution $\pi$ over $\{1,\dots,K\}$
\begin{equation}
\sum_{t=1}^T \E_{t-1}\Big[\E_\pi[r_{k,t}]\Big]
\le  \E_\pi\bigg[ 4\lambda_k  \sum_{t=1}^T   \E_{t-1}[r_{k,t}^2] + \frac{9}{2} \lambda_k E^2 x
  + \frac{4a_k+x}{\lambda_k } +  6 a_k E\bigg] + \frac{2}{T}.
    \label{eq:riskrisklambdak}
\end{equation}
Now, we are ready to apply Assumption~\ref{ass:a2} in order to cancel the sum in the right-hand side. Assumption~\ref{ass:a2} ensures that for any time $t \geq 1$
\[
  \E_{t-1}\big[\ell_t(\hat \theta_{t-1}) - \ell_t(\theta_k) \big] \leq \E_{t-1}[r_{k,t}] -\left( \alpha \E_{t-1}[r_{k,t}^2] \right)^{1/\beta} \,.
\]
Therefore, summing over $t=1,\dots,T$ and using the preceding inequality with probability at least $1-2e^{-x}$
\begin{align}
& \E_\pi\bigg[ \sum_{t=1}^T \E_{t-1}\big[\ell_t(\hat \theta_{t-1}) -  \ell_t(\theta_k) \big]\bigg]
  \leq \E_\pi\left[\sum_{t=1}^T\E_{t-1}[ r_{k,t}] - \left(\alpha  \E_{t-1}[r_{k,t}^2] \right)^{1/\beta}\right] \nonumber \\
  &  \qquad \le   \E_\pi\bigg[ 4 \lambda_k  \sum_{t=1}^T    \E_{t-1}[r_{k,t}^2]  -  \sum_{t=1}^T \Big(\alpha   \E_{t-1}[r_{k,t}^{2}]  \Big)^{1/\beta}  + \frac{9}{2}\lambda_k E^2 x + \frac{ 4a_k+x}{\lambda_k } + 6a_kE\bigg] + \frac{2}{T}. \label{eq:regret_eta_alpha}
\end{align}

Now, we use Young's inequality (see Lemma~\ref{lem:young}) again to cancel the two sums in the right-hand side. Let $\gamma >0$ to be fixed later by the analysis. Using $a = \E_{t-1}[r_{k,t}^2]/\gamma$, $b=\gamma$, $p=1/\beta$, and $q = 1/(1-\beta)$,  it yields
\[
  \E_{t-1} [ r_{k,t}^2]  \leq \frac{\beta \big( \E_{t-1}[r_{k,t}^{2}]\big)^{1/\beta}}{\gamma^{1/\beta}} + \big(1 - \beta\big)\gamma^{1/(1-\beta)} \,.
\]
Thus,
\[
  \lambda_k  \E_{t-1}[r_{k,t}^2] \leq \frac{\lambda_k \beta \big( \E_{t-1}[r_{k,t}^{2}]\big)^{1/\beta}}{\gamma^{1/\beta}} + \lambda_k \big(1 - \beta\big)\gamma^{1/(1-\beta)} \,.
\]
The choice $\gamma = (4\lambda_k\beta)^{\beta} /\alpha$ yields $4\lambda_k\beta/\gamma^{1/\beta} = \alpha^{1/\beta}$, which entails
\begin{align}
  4 \lambda_k   \E_{t-1}[r_{k,t}^2] - \big(\alpha \E_{t-1}[r_{k,t}^2]\big)^{1/\beta}
    & \leq  4 \lambda_k \big(1 - \beta\big) \gamma^{1/(1-\beta)} \nonumber \\
    & = 4 \lambda_k \big(1 - \beta\big) \Big(\frac{(4\lambda_k\beta)^\beta}{\alpha}\Big)^{1/(1-\beta)} \nonumber  \\
    & =  4 \big(1 - \beta\big) (4\beta)^{\beta/(1-\beta)}\Big(\frac{\lambda_k}{\alpha}\Big)^{1/(1-\beta)} \nonumber \\
        & \le   4 \Big(\frac{4\lambda_k}{\alpha}\Big)^{1/(1-\beta)} 
        \label{eq:applyYoungs} \,.
\end{align}
Summing over $t$ and substituting into Inequality~\eqref{eq:regret_eta_alpha}, we get
\begin{equation}
  \E_\pi\left[ \sum_{t=1}^T \E_{t-1}\big[\ell_t(\hat \theta_{t-1}) - \ell_t(\theta_k) \big] \right]   
  \le  E_\pi\biggl[ \underbrace{4\Big(\frac{4\lambda_k}{\alpha}\Big)^{1/(1-\beta)} T +\frac{   4a_k+x}{\lambda_k } }_{=: R_k} +\frac{9}{2}\lambda_k E^2x + 6a_kE  \biggr]  +\frac{2}{T}\,. \label{eq:regret_eta_alpha2}
\end{equation}
We optimize $\lambda_k$ by equalizing the two main terms of $R_k$:
\[
  4 \Big(\frac{4\lambda_k}{\alpha}\Big)^{1/(1-\beta)} T = \frac{4a_k + x}{\lambda_k}
    \Leftrightarrow \lambda_k = \left(\frac{4a_k + x}{4T}\right)^{\frac{1-\beta}{2-\beta}} \left(\frac{\alpha}{4}\right)^{\frac{1}{2-\beta}} \,.
\]
We express $R_k$ in termes of  $\lambda_k$ using this identity
\[
  \frac{R_k}{T} = 2 \frac{4a_k+x}{\lambda_k T} = 2\left(\frac{4a_k+x}{\alpha T}\right)^{\frac{1}{2-\beta}} 4^{\frac{1-\beta}{2-\beta}} \leq 4\left(\frac{16a_k+4x}{\alpha T}\right)^{\frac{1}{2-\beta}} \,.
\]
The choice $\lambda_k = 1/(2E)$ would give
\[
  \frac{R_T}{T}  \leq 4 \Big(\frac{4\lambda_k}{\alpha}\Big)^{1/(1-\beta)} + \frac{4a_k + x}{T\lambda_k} \leq \frac{(4a_k + x) E}{T} \,.
\]
So that we can assume $\lambda_k \leq 1/(2E)$ and 
\[
  \frac{R_T}{T} \leq 4\left(\frac{16a_k+4x}{\alpha T}\right)^{\frac{1}{2-\beta}} + \frac{(4a_k + x) E}{T}
\]
Substituting into Inequality~\eqref{eq:regret_eta_alpha2} and upper-bounding $\lambda_kE^2 \leq E/2$, gives
\begin{equation*}
  \frac{1}{T} \E_\pi\left[ \sum_{t=1}^T \E_{t-1}\big[\ell_t(\hat \theta_{t-1}) - \ell_t(\theta_k) \big] \right]   
  \le  E_\pi\biggl[ 4\left(\frac{16a_k+4x}{\alpha T}\right)^{\frac{1}{2-\beta}} + \frac{(10a_k+4x)E}{T}  \biggr]  +\frac{2}{T^2}\,. 
\end{equation*}
Replacing $a_k = \log(\pi_k/\pi_{k,0})+\log\log(ET^2)$ concludes the proof.

\subsection{Proof of Theorem~\ref{thm:slowrate}}

\label{app:proof_slowrate}
   We denote by $\theta_1,\dots,\theta_K$ the elements of $\Theta_0$. We recall that we use a particular case of Algorithm~\ref{alg:BOA}.  We can thus follow the proof of Theorem~\ref{thm:BOA_fastrate} and start from Inequality~\eqref{eq:firstregret}. We apply it to a Dirac distributions $\pi$ on $\{1,\dots, K\}$. We get that for any $1\le k \leq K$, for any $\lambda_k >0$,
   \begin{equation}
    \label{eq:firsteq}
      \sum_{t=1}^T  r_{k,t} \leq  4 \sqrt{a \sum_{t=1}^T r_{k,t}^2}  + 6aE + \frac{2}{T} \,.
   \end{equation}
  where $a:= \log(K) + \log \log (ET^2)$ and where we remind the notation of the linearized instantaneous regret $r_{k,t} = \nabla \ell_t(\hat  \theta_{t-1})^\top (\hat  \theta_{t-1} - \theta_k)$ for $1\leq k\leq K$.

   Let $\theta^*  \in \R^d$, let $\epsilon := D(\theta^*,\Theta_0)$ and $k^* \in \{1\leq k\leq K\}$ such that $\|\theta^* - (1-\epsilon)\theta_{k^*}\|_1 \leq \epsilon$.  Then it exists $\tilde \theta$ with $\|\tilde \theta\|_1\leq 1$ such that 
   \begin{equation}
    \label{eq:decomposition}
    \theta^* = (1-\epsilon)\theta_{k^*} +\epsilon \tilde \theta \,.
    \end{equation} 
    Since $\{\theta \in \cB_1:\|\theta\|_1=1, \|\theta\|_0=1\} \subset \Theta_0$, we can write $\tilde \theta$ as a combination of elements of $\Theta_0$. Hence, from~\eqref{eq:decomposition}, it exists a distribution $\pi=(\pi_1,\ldots, \pi_{K}) \in \Delta_{K}$ such that
 \[
    \theta^* = \sum_{k=1}^{K} \pi_{k}  \theta_{k}   \quad \mbox{and}\quad 1-\pi_{k^*} \le \epsilon.
 \]
 Denoting $r_t:= \nabla \ell_t(\hat  \theta_{t-1})^\top(\hat  \theta_{t-1} - \theta^*)$, we thus get 
 \begin{multline*}
    r_t :=  \nabla \ell_t(\hat  \theta_{t-1})^\top(\hat  \theta_{t-1} - \theta^*) =   \nabla \ell_t(\hat \theta_{t-1})^\top \Big(\hat \theta_{t-1} -  \sum_{k=1}^{K} \pi_k \theta_k \Big) \\
    =  \nabla \ell_t(\hat \theta_{t-1})^\top(\hat \theta_{t-1} -  \E_{k \sim \pi} [\theta_k] ) =  \E_{k \sim \pi} \Big[ r_{k,t}\Big]\,,
 \end{multline*}
and integrating Inequality~\eqref{eq:firsteq} with respect to $\pi$, we obtain
  \begin{align}
  \sum_{t=1}^T r_t  & \leq    \E_{k \sim \pi}\bigg[ 4 \sqrt{ a\sum_{t=1}^T \big(\nabla \ell_t(\hat \theta_{t-1})^\top(\hat \theta_{t-1} -  \theta^*  +  \theta^* - \theta_k)\big)^2}  \bigg] + \frac{2}{T} + 6aE \nonumber  \\
        &  \leq 4  \sqrt{a \sum_{t=1}^T r_t^2 }+  4  \E_{k \sim \pi}\bigg[\sqrt{a \sum_{t=1}^T \big(\nabla \ell_t(\hat \theta_{t-1})^\top(\theta^*  - \theta_k)\big)^2 }\bigg] + \frac{2}{T} + 6aE\,.\label{eq:someeq}
  \end{align}
  Let us upper bound the second term of the right hand side.
\begin{align}
 \E_{k \sim\pi}\bigg[ & \sqrt{ \sum_{t=1}^T  \big(\nabla \ell_t(\hat \theta_{t-1})^\top(\theta^*  - \theta_k)\big)^2} \bigg] \nonumber \\
 &\le \sqrt{ \sum_{t=1}^T \|\nabla \ell_t(\hat \theta_{t-1})\|_\infty^2} \sum_{k=1}^{K} \pi_k \|\theta^* - \theta_{k}\|_1 \nonumber \\
 &\le \sqrt{ \sum_{t=1}^T \|\nabla \ell_t(\hat \theta_{t-1})\|_\infty^2}\bigg(\pi_{k^*}\| \theta^*- \theta_{k^*}\|_1  +(1-\pi_{k^*})\max_{1\le k\le K}\| \theta^*- \theta_{k}\|_1\bigg)\nonumber \\
&\le \sqrt{ \sum_{t=1}^T \|\nabla \ell_t(\hat \theta_{t-1})\|_\infty^2 }\bigg(\pi_{k^*}\| \theta^*- \theta_{k^*}\|_1+2 (1-\pi_{k^*}) \bigg) \,, \label{eq:previouseq}
\end{align}
where the last inequality is because $\|\theta^* - \theta_{k}\|_1 \leq \|\theta_k\|_1 + \|\theta^*\|_1 \leq 2$. We also have from the definition of $\theta^*$ (see before ~\eqref{eq:decomposition})
\[
\| \theta^*- \theta_{k^*}\|_1 \leq \| \theta^*- (1-\epsilon)\theta_{k^*}+\epsilon \theta_{k^*}\|_1 \leq \| \theta^*- (1-\epsilon)\theta_{k^*}\|_1 +\epsilon \|\theta_{k^*}\|_1 \leq 2\epsilon .
\]
Therefore, substituting into~\eqref{eq:previouseq} we get
\[
 \E_{k \sim\pi}\bigg[ \sqrt{ \sum_{t=1}^T  \big(\nabla \ell_t(\hat \theta_{t-1})^\top(\theta^*  - \theta_k)\big)^2} \bigg] \\
\le 4 \epsilon \sqrt{ \sum_{t=1}^T \|\nabla \ell_t(\hat \theta_{t-1})\|_\infty^2 } = 4\epsilon  \bar G_T \sqrt{T}  ,
\]
where $\bar G_T :=  \sqrt{\frac{1}{T} \sum_{t=1}^T \|\nabla \ell_t(\hat \theta_{t-1})\|_\infty^2 } \leq G $.

Therefore, substituting into Inequality~\eqref{eq:someeq}, we have
\[
  \sum_{t=1}^T r_t \leq 4 \sqrt{a \sum_{t=1}^T r_t^2}  +  16 \epsilon \bar G_T \sqrt{aT} + \frac{2}{T} + 6aE\,,
\]
which yields by Young's inequality for any $\lambda >0$
\begin{equation}
  \sum_{t=1}^T r_t \leq \lambda \sum_{t=1}^T r_t^2  + \frac{4 a}{\lambda} +  \underbrace{16 \epsilon \bar G_T \sqrt{aT} + \frac{2}{T} + 6aE}_{=: z}\,.
 \label{eq:ineqslow}
\end{equation}
Now, we recognize an inequality similar to Inequality~\eqref{eq:regretlambdak}. There only are a few technical differences which do not matter in the analysis: we consider here a Dirac distribution $\pi$ on the comparison parameter $\theta^*$ and we have some additional rest terms that we denote by $z := 16\epsilon  \bar G_T \sqrt{aT} + \frac{2}{T} + 6aE$ for simplicity. We can then follow the lines of the proof of Theorem~\ref{thm:BOA_fastrate} after Inequality~\eqref{eq:regretlambdak}
\begin{eqnarray}
  \sum_{t=1}^T \E_{t-1}[r_t] 
    & \stackrel{\text{Thm 4.1 of \cite{Wintenberger2014}}}{\leq} & \sum_{t=1}^T r_t + \lambda \sum_{t=1}^T r_t^2 + \frac{x}{\lambda} \nonumber \\
    & \stackrel{\eqref{eq:ineqslow}}{\leq} & 2\lambda \sum_{t=1}^T r_t^2  + \frac{4 a + x}{\lambda} +  z \nonumber \\
    & \stackrel{\text{Thm 9 of \cite{GaillardWintenberger2017}}}{\leq} & 4\lambda \sum_{t=1}^T \E_{t-1}[r_t^2]  + \frac{4 a + x}{\lambda}  + \frac{9}{2}\lambda E^2 x+  z . \label{eq:beforeapplyingA2}
\end{eqnarray}
Using Assumption~\ref{ass:a2} then yields
\begin{eqnarray*}
  \sum_{t=1}^T \E_{t-1}\big[\ell_t(\hat \theta_{t-1}) - \ell_t(\theta^*)\big] 
    & \leq & \sum_{t=1}^T \E_{t-1}[r_t] - \big(\alpha \E_{t-1}[r_t^2]\big)^{1/\beta} \\
    & \stackrel{\eqref{eq:beforeapplyingA2}}{\leq} &  4\lambda \sum_{t=1}^T \E_{t-1}[r_t^2]  - \big(\alpha \E_{t-1}[r_t^2]\big)^{1/\beta} + \frac{4 a + x}{\lambda}  + \frac{9}{2}\lambda E^2 x+  z \\
    & \stackrel{\eqref{eq:applyYoungs}}{\leq} & 4 \Big(\frac{4\lambda}{\alpha}\Big)^{1/(1-\beta)} + \frac{4 a + x}{\lambda}  + \frac{9}{2}\lambda E^2 x+  z .
\end{eqnarray*}

This yields an inequality similar to Inequality~\eqref{eq:regret_eta_alpha2}. Optimizing in $\lambda >0$, as we did for Inequality~\eqref{eq:regret_eta_alpha2} gives: 
\[
  \lambda = \min\left\{\frac{1}{2E}, \left(\frac{4a + x}{4T}\right)^{\frac{1-\beta}{2-\beta}} \left(\frac{\alpha}{4}\right)^{\frac{1}{2-\beta}}\right\},
\]
and 
\[
  \frac{1}{T} \sum_{t=1}^T \E_{t-1}\big[\ell_t(\hat \theta_{t-1}) - \ell_t(\theta^*)\big] \leq 4 \left( \frac{16a+4x}{\alpha T}\right)^{\frac{1}{2-\beta}} + \frac{(4a+x)E}{T} + \frac{9Ex}{4T} + \frac{z}{T} \,.
\]
where we recall that $a = \log(K) + \log \log (ET^2)$, $z = 16\epsilon  \bar G_T \sqrt{aT} + \frac{2}{T} + 6aE$ and $\bar G_T :=  \sqrt{\frac{1}{T} \sum_{t=1}^T \|\nabla \ell_t(\hat \theta_{t-1})\|_\infty^2 } \leq G $. Replacing $z$ with its definition and simplifying yields
\begin{equation}
  \frac{1}{T} \sum_{t=1}^T \E_{t-1}\big[\ell_t(\hat \theta_{t-1}) - \ell_t(\theta^*)\big] \leq 4 \left( \frac{16a+4x}{\alpha T}\right)^{\frac{1}{2-\beta}} + \frac{(10a+4x)E}{T} +  16 \epsilon  \bar G_T\sqrt{\frac{a}{T}} + \frac{2}{T^2} \,.
  \label{eq:thmslowrateexact}
\end{equation}
Keeping the main terms only and replacing $\epsilon := D(\theta,\Theta_0)$ concludes the proof.

\subsection{Proof of Lemma~\ref{lem:rlarge}}
\label{app:rlarge}
Let $\pi : = \|\theta' - \theta\|_1/(\|\theta' - \theta\|_1+1-\|\theta\|_1)$. Then, thanks to the triangular inequality, we have
\begin{multline*}
  \big\|\theta - (1-\pi)\theta'\big\|_1 = \big\|(1-\pi) (\theta-\theta') + \pi \theta\big\|_1 \leq (1-\pi)\|\theta - \theta'\|_1 + \pi \|\theta\|_1 \\
    = \frac{(1-\|\theta\|_1)\|\theta - \theta'\|_1 + \|\theta - \theta'\|_1 \|\theta\|_1}{\|\theta-\theta'\|_1 + 1 - \|\theta\|_1} = \pi \,.
\end{multline*}
The Definition~\ref{def:distance} of $D(\theta,\theta')$ concludes the proof.

\subsection{Proof of Lemma~\ref{lem:supportmin}}
\label{app:supportmin}
Denote $\pi := 1- \min_{1\leq i\leq d} |\theta_i|/|\theta'_i|$. Then, for any $1\leq i \leq d$, $|\theta_i| \geq (1-\pi) |\theta'_i|$. Because $\theta'_i$ and $\theta_i$ have same signs, this yields $\big| \theta_i - (1-\pi) \theta'_i \big| = \big| \theta_i\big| - (1-\pi) \big|\theta'_i\big|$ for all $1\leq i\leq d$. Summing over $i=1,\dots,d$, entails
\begin{multline}
	\big\| \theta - (1-\pi) \theta'\big\|_1 = \sum_{i=1}^d \Big| \theta_i - (1-\pi) \theta'_i \Big| = \sum_{i=1}^d \big| \theta_i\big| - (1-\pi) \big|\theta'_i\big| \\
	= \|\theta\|_1 - (1-\pi) \|\theta'\|_1 \stackrel{\|\theta'\|_1\geq \|\theta\|_1}{\leq} \pi\|\theta\|_1 \leq \pi.
	\label{eq:D1}
\end{multline}
Therefore, the Definition \ref{def:distance} of $D(\theta,\theta')$ concludes the proof of the first inequality. Now, let $1\leq i\leq d$, if $|\theta'_i| \leq |\theta_i|$ then $1-|\theta_i|/|\theta'_i| \leq 0$ and the second inequality holds. Otherwise, we have
\[
  1 - \frac{|\theta_i|}{|\theta'_i|} = \frac{|\theta'_i|-|\theta_i|}{|\theta'_i|} \stackrel{|\theta'_i| \geq |\theta_i|}{=} \frac{|\theta'_i - \theta_i|}{|\theta'_i|} \stackrel{|\theta'_i| \geq |\theta_i|}{\leq} \frac{|\theta'_i - \theta_i|}{|\theta_i|} \leq \frac{\|\theta' - \theta\|_\infty}{\Delta} \,,
\]
which concludes the proof of the Lemma.

\subsection{Proof of Lemma~\ref{lem:softthreshold}}
\label{app:softthreshold}

Let $\theta, \theta' \in \cB_1$ such that $\|\theta - \theta'\|_\infty \leq \epsilon$. First, we check that $\tilde \theta$ satisfies the assumptions of Lemma~\ref{lem:supportmin}. 
Since $\|\theta' - \theta\|_\infty \leq \epsilon$, for all coordinates $1\leq i\leq d$, we have $S_\epsilon(\theta')_i = 0$ or $\sign(S_\epsilon(\theta'))_i = \sign(\theta_i)$. Therefore,
$\sign(\tilde \theta_i) = \sign(S_\epsilon(\theta')_i) \in \{0,\sign(\theta_i)\}$. Furthermore,
\begin{equation}
  \|S_\epsilon(\theta')\|_1
  \geq \sum_{i \in \Supp(\theta)} \big|S_\epsilon(\theta')_i\big|
  \geq \sum_{i \in \Supp(\theta)} \big(\big|\theta'_i\big| - \epsilon \big)
  \geq \sum_{i \in \Supp(\theta)} \big(\big|\theta_i\big| - 2\epsilon \big) \geq \|\theta\|_1 - 2d_0\epsilon \,.
  \label{eq:norme1small}
\end{equation}
If $S_\epsilon(\theta') = 0$, then $\|\tilde \theta\|_1 = 0$ and $\|\theta\|_1 \leq 2d_0\epsilon$ so that $D(\theta,\tilde \theta) \leq 1\leq 2d_0\epsilon/\|\theta\|_1$. Therefore, we can assume from now that $S_\epsilon(\theta') \neq 0$. By definition of $\tilde \theta$, Inequality~\eqref{eq:norme1small} yields $\|\tilde \theta\|_1 = \big(\|S_\epsilon(\theta')\|_1 + 2d_0 \epsilon \big)  \wedge 1 \geq \|\theta\|_1$. Then $\tilde \theta$ satisfies the assumptions of Lemma~\ref{lem:supportmin}, which we can apply
\begin{equation}
  D(\theta,\tilde \theta)
     \leq 1- \min_{1\leq i\leq d} \frac{|\theta_i|}{|\theta'_i|}
     = \max_{i \in \Supp(\tilde \theta)} \frac{|\tilde \theta_i| - |\theta_i|}{|\tilde \theta_i|} \,.{}
     \label{eq:D1i}
\end{equation}

We consider two cases:
\begin{itemize}
  \item $\|S_\epsilon(\theta')\|_1 \geq 1-2d_0\epsilon$ in which case for $i \in \Supp(\tilde\theta)$
  \[
    \tilde \theta_i = \frac{S_\epsilon(\theta')_i}{\|S_\epsilon(\theta')\|_1} = \frac{(|\theta'_i| - \epsilon)\sign(\theta_i')}{\|S_\epsilon(\theta')\|_1}
  \]
  so that $|\tilde \theta_i| = (|\theta'_i|-\epsilon)/\|S_\epsilon(\theta')\|_1$ and upper-bounding $-|\theta_i| \leq -|\theta'_i| - \epsilon$ we get
  \begin{multline*}
      \frac{|\tilde \theta_i| - |\theta_i|}{|\tilde \theta_i|} =
      \frac{|\theta'_i| - \epsilon - |\theta_i| \|S_\epsilon(\theta')\|_1}{|\theta'_i|-\epsilon} \leq
      \frac{|\theta'_i| - \epsilon - (|\theta'_i| - \epsilon) \|S_\epsilon(\theta')\|_1}{|\theta'_i|-\epsilon} \\
      \leq 1- \|S_\epsilon(\theta')\|_1 \leq 2d_0\epsilon \leq \frac{2d_0\epsilon}{\|\theta\|_1} \,.
  \end{multline*}
  Substituting into Inequality~\eqref{eq:D1i} concludes this case.
  \item Otherwise $\|S_\epsilon(\theta')\|_1 \leq 1-2d_0\epsilon$ and for $i \in \Supp(\tilde\theta) = \Supp(S_\epsilon(\theta'))$
  \[
    |\tilde \theta_i| = |S_\epsilon(\theta')_i| \Big(1 + \frac{2d_0\epsilon}{\|S_\epsilon(\theta')\|_1}\Big) = (|\theta'_i| - \epsilon) \Big(1 + \frac{2d_0\epsilon}{\|S_\epsilon(\theta')\|_1}\Big) \,,
  \]
  which implies upper-bounding $-|\theta_i| \leq -|\theta'_i| - \epsilon$,
  \begin{align*}
      \frac{|\tilde \theta_i| - |\theta_i|}{|\tilde \theta_i|}
      & =
      \frac{(|\theta'_i| - \epsilon) \Big(1 + \frac{2d_0\epsilon}{\|S_\epsilon(\theta')\|_1}\Big)  - |\theta_i| }{\big(|\theta'_i|-\epsilon\big) \Big(1 + \frac{2d_0\epsilon}{\|S_\epsilon(\theta')\|_1}\Big)}  \\
      & \leq
      \frac{\big(|\theta'_i| - \epsilon\big)  \frac{2d_0\epsilon}{\|S_\epsilon(\theta')\|_1}}{\big(|\theta'_i|-\epsilon\big) \Big(1 + \frac{2d_0\epsilon}{\|S_\epsilon(\theta')\|_1}\Big)} \\
      & = \frac{2d_0\epsilon}{\|S_\epsilon(\theta')\|_1 + 2d_0\epsilon}  \\
      & \leq  \frac{2d_0\epsilon}{\|\tilde \theta\|_1} \leq \frac{2d_0\epsilon}{\|\theta\|_1} \,.
  \end{align*}
\end{itemize}
Substituting the obtained bounds in each cases into Inequality~\eqref{eq:D1i} concludes the proof.

\subsection{Proof of Theorem~\ref{thm:adversarial}}

For technical reasons, we perform the proof for $\theta \in \cB_{1/2}$ only. However, optimization on $\cB_1$ can be obtained by renormalizing the losses considering $\ell_t(2 \theta)$ instead of $\ell_t$. We leave this generalization for the reader. For simplicity, we also assume that $T = 2^I-1$. Let $\theta \in \cB_{1/2}$ and denote $d_0 = \|\theta\|_0$.

\emph{Part 1 ($\tilde \cO(\sqrt{T})$ regret -- logarithmic dependence on $d_0$ and $d$)}
First, we prove the slow rate bound obtained by Algorithm~\ref{alg:BOA}. Let $i \geq 0$. Denote by $\theta_1,\dots,\theta_{3d}$ the $3d$ elements of $\Theta^{(i)}$. 
For any distribution $\pi \in \Delta_{3d}$ over $\Theta^{(i)}$, we have from Inequality~\eqref{eq:firstregret}:
  \begin{equation}
    \label{eq:firstregret1}
     \sum_{t=t_i}^{t_{i+1}-1} \E_{k \sim \pi}[r_{k,t}] \leq  \E_{k \sim \pi}\left[ 4 \sqrt{a_k V_k}  + 6 a_k E\right] + \frac{2}{T}.
  \end{equation}
where we recall $r_{k,t} \leq \nabla\ell_t(\hat \theta_{t-1})^\top (\hat \theta_{t-1} - \theta_k)$, $a_k := \log(\pi_k/\pi_{k,0}) + \log \log (ET^2) \leq \log(3d) + \log\log (ET^2) =: a$ and $V_k \leq \sum_{t=t_i}^{t_{i+1}-1} r_{k,t}^2 \leq t_i G^2$. Let $\pi$ such that $\theta = \sum_{k=1}^{3d} \pi_k \theta_k$, then thanks to the convexity assumption on the losses, we have
\[
  \ell_t(\hat \theta_{t-1}) - \ell_t(\theta) \leq \nabla \ell_t(\hat \theta_{t-1})^\top (\hat \theta_{t-1} - \theta_k) =  \E_{k \sim \pi}[r_{k,t}] \,.
\] 
Therefore, Inequality~\eqref{eq:firstregret1} yields
\[
  \sum_{t=t_i}^{t_{i+1}-1} \ell_t(\hat \theta_{t-1}) - \ell_t(\theta) \leq 4 G \sqrt{ a t_i} + 6 E a + \frac{2}{T} \,.
\]
Summing over $i=0,\dots,j-1$ and substituting $t_i = 2^i$ we get for any $j\geq 1$:
\begin{equation}
  \Reg_j(\theta) := \sum_{t=1}^{t_{j}-1} \ell_t(\hat \theta_{t-1}) - \ell_t(\theta) 
    \leq 4 G \sqrt{ a} \sum_{i=0}^{j-1} 2^{i/2}  + \underbrace{6 E a j + \frac{2j}{T}}_{ =: z} \ 
     \leq 10 G \sqrt{ a}2^{j/2}  + z \,,
    \label{eq:regretslow}
\end{equation}
where $a = \log(2d) + \log\log (ET^2)$. In particular for $j = I \lesssim \log T$ we obtain the first inequality stated by the theorem:
\[
  R_T(\theta) \leq  \cO \left(G \sqrt{\frac{\log d + \log\log (ET)}{T}}\right) \,.
\] 

\emph{Part 2 ($\tilde \cO(T^{1/4})$ regret -- logarithmic dependence on $d$)} 
We prove by induction the second bound of the Theorem: that for some $c>0$ and all $j\geq 0$, we have
\begin{equation}
  \Reg_j(\theta) \leq  48 \frac{a d_0 c^2 G^2}{\mu} + j \frac{caG^2}{\mu} + 16\sqrt{5} c \sqrt{\frac{d_0 \left(G \sqrt{a}\right)^{3}}{\mu}}  \sum_{k=0}^j   2^{-\frac{3j}{4}}  \,.
  \label{eq:induction2}
\end{equation}
Indeed, decomposing the cumulative regret, we have
\[
  \Reg_{j+1}(\theta) = \Reg_{j}(\theta) + \sum_{t=t_j}^{t_{j+1}-1} \ell_t(\hat \theta_{t-1}) - \ell_t(\theta) \,.
\]
Note that Assumption~\ref{ass:a2} is satisfied with $\beta = 1$, $\alpha = \mu/(2G^2)$ and without the expectation $\E_t$. It is worth to notice that the transformation of the second-order term into a cumulative risk performed in~\eqref{eq:beforeapplyingA2} was not needed here since Assumption~\ref{ass:a2} holds on the losses without the expectation $\E_t$. Therefore, the result of Theorem~\ref{thm:slowrate}, that we can apply from time instance $t_j = 2^j$ to $t_{j+1}-1$, holds almost surely with $x=0$, $\beta = 1$ and $\alpha = \mu/(2G^2)$. We get that it exists some constant $c>0$ such that
\[
  \sum_{t=t_j}^{t_{j+1}-1} \ell_t(\hat \theta_{t-1}) - \ell_t(\theta) \leq c G  D(\theta, [\theta^*_{j}]_{d_0}) \sqrt{a 2^j} + \frac{caG^2}{\mu} \,,
\]
with $a = \log (3d) + \log \log (ET^2)$.
Replacing into the preceding inequality, it yields
\begin{equation}
   \Reg_{j+1}(\theta) \leq \Reg_{j}(\theta) +  c G  D(\theta, [\theta^*_{j}]_{d_0}) \sqrt{a 2^j}  + \frac{c a G^2}{\mu} 
  \label{eq:regretboundi}
\end{equation}
Because $\theta \in \cB_{1/2}$, we obtain from Lemma~\ref{lem:rlarge} 
\begin{multline*}
  D\big(\theta,[\theta^*_{j}]_{d_0}\big) 
     \stackrel{\text{(Lem.~\ref{lem:rlarge})}}{\leq} 2 \big\|\theta - [\theta^*_{j}]_{d_0}\big\|_1   \stackrel{\|\theta\|_0 = \|[\theta^*_{j}]_{d_0}\|_0 = d_0}{\leq} 2 \sqrt{2d_0} \big\|\theta - [\theta^*_{j}]_{d_0}\big\|_2  \nonumber \\
     \leq 2 \sqrt{2d_0} \big(\big\|\theta - \theta^*_{j} \big\|_2 + \big\| \theta^*_{j}  -  [\theta^*_{j}]_{d_0}\big\|_2 \big)\,.
\end{multline*}
By definition of the hard threshold, for any $\theta$ such that $\|\theta\|_0 = d_0$, we have
\begin{equation*}
  \big\| \theta^*_{j}  -  [\theta^*_{j}]_{d_0}\big\|_2 \leq \big\| \theta^*_{j}  -  \theta \big\|_2 \,.
\end{equation*}
Therefore, plugging into the previous inequality
\begin{equation}
  D\big(\theta,[\theta^*_{j}]_{d_0}\big)  \leq 4\sqrt{2d_0} \big\|\theta - \theta^*_{j} \big\|_2  \,.
  \label{eq:metricnorme2}
\end{equation}
But because the losses are $\mu$-strongly convex, the average loss over several rounds is also $\mu$-strongly convex. And since $\theta^*_j :=  \argmin_{\theta \in \cB_{1/2}} \sum_{t=1}^{t_j-1} \ell_t(\theta)$, we have for all $\theta \in \cB_{1/2}$
\begin{align}
  \mu \big\|\theta - \theta^*_{j}\big\|_2^2 
    & \leq \frac{1}{2^j-1} \sum_{t=1}^{t_j-1} \ell_t(\theta) - \ell_t(\theta^*_j) \label{eq:relaxedstrongconvexity}\\
  & = \frac{\Reg_j(\theta^*_j)-\Reg_j(\theta)}{2^j-1} \leq  \frac{\Reg_j(\theta^*_j)-\Reg_j(\theta)}{2^{j-1}} \,.
  \nonumber 
\end{align}
Thus, from Inequality~\eqref{eq:metricnorme2}, we obtain
\[
  D\big(\theta,[\theta^*_{j}]_{d_0}\big)  \leq  8\sqrt{\frac{d_0 \big(\Reg_j(\theta^*_j)-\Reg_j(\theta)\big)}{\mu 2^j} }\,.
\]
Plugging into Inequality~\eqref{eq:regretboundi} gives 
\begin{equation}
  \Reg_{j+1}(\theta) \leq \Reg_{j}(\theta) + 8 c G   \sqrt{ \frac{ad_0}{\mu} \big(\Reg_j(\theta^*_j)-\Reg_j(\theta)\big)}  + \frac{caG^2}{\mu}\,.
  \label{eq:regretinduction}
\end{equation}
We can upper-bound $\Reg_j(\theta^*_j)$ using Inequality~\eqref{eq:regretslow}. This entails
\[
  \Reg_{j+1}(\theta) \leq \Reg_{j}(\theta) + 8 c G   \sqrt{ \frac{ad_0}{\mu} \big(10 G \sqrt{ a}2^{j/2}   + z - \Reg_j(\theta)  \big)}  + \frac{caG^2}{\mu}\,.
\]
Now we have an inequality of the form
\[
   \Reg_{j+1}(\theta) \leq \Reg_j(\theta) + x_1 \sqrt{x_2 - \Reg_j(\theta)} + x_3
\]
with $x_1 = 8cG\sqrt{ad_0/\mu}$, $x_2 = 10G \sqrt{a}2^{j/2} + z$ and $x_3 = {(caG^2)}/{\mu}$. If $\Reg_j(\theta) \geq 0$, $\Reg_{j+1}(\theta)$ is increased by at most $x_1\sqrt{x_2} + x_3$. Otherwise $\Reg_j(\theta) \leq 0$ and the right-hand side is at most $3 x_1^2/4 + x_3$ (considering the maximum over $\Reg_j(\theta)\leq 0$). Therefore,
\begin{align}
  \Reg_{j+1}(\theta) 
    & \leq \max\big\{ 3 x_1^2/4, (\Reg_j(\theta))_+ + x_1\sqrt{x_2}\big\}  + x_3 \,. \nonumber \\
    &  =  \max\left\{ 48 \frac{a d_0 c^2 G^2}{\mu}, \  (\Reg_j(\theta))_+ + 16\sqrt{5} c \sqrt{\frac{d_0}{\mu}} \left(G \sqrt{\frac{a}{2^j}}\right)^{3/2}\right\} + \frac{caG^2}{\mu}\,.
  \label{eq:inductionplus}
\end{align}
This concludes the induction, using the hypothesis~\eqref{eq:induction2}.  
In particular, considering $j = I = \log_2(T-1)$, we proved that 
\[
  R_T(\theta)  = \frac{\Reg_{I}(\theta)}{T} \leq \cO \left( \sqrt{\frac{d_0}{\mu}} \left(G \sqrt{\frac{\log d + \log \log (ET)}{T}}\right)^{\frac{3}{2}} \right) \,.
\]

\bigskip
\emph{Part 3. ($\tilde \cO(1)$ regret -- square root dependence on $d$)} Now, we prove a faster rate but at the price of a square root dependence in the total dimension $d$. The proof follows the same lines as the preceding part except that one changes the induction hypothesis and that one uses it to bound the regret of $\theta^*_j$. We prove by induction: it exists $c_0>0$ such that for any $\theta \in \cB_{1/2}$
\[
  \Reg_j(\theta) \leq j \frac{ac_0 \sqrt{\|\theta\|_0 d} G^2}{\mu T}   
\]
where $a = \log (3d) + \log \log (ET^2)$.
We start from Inequality~\eqref{eq:regretinduction} obtained in Part~2: 
\begin{equation*}
  \Reg_{j+1}(\theta) \leq \Reg_{j}(\theta) + 8 c G   \sqrt{ \frac{ad_0}{\mu} \big(\Reg_j(\theta^*_j)-\Reg_j(\theta)\big)}  + \frac{caG^2}{\mu}\,.
\end{equation*}
Now, instead of upper-bounding $\Reg_j(\theta^*_j)$  using Inequality~\eqref{eq:regretslow}, we  use the induction hypothesis itself. Since $\theta^*_j$ is not necessarily sparse, we have
\[
  \Reg_j(\theta^*_j)  \leq  j \frac{ac_0d  G^2}{\mu}  \,,
\]
which entails
\[
  \Reg_{j+1}(\theta) \leq \Reg_{j}(\theta) + 8 c G   \sqrt{ \frac{ad_0}{\mu} \Big(j \frac{ac_0d  G^2}{\mu }   -\Reg_j(\theta)\Big)}  +\frac{caG^2}{\mu} \,.
\]
We obtain a regret bound of the same form than in Part~2:
\[
   \Reg_{j+1}(\theta) \leq \Reg_j(\theta) + x_1 \sqrt{x_2 - \Reg_j(\theta)} + x_3,
\]
with $x_1 = 8cG \sqrt{a d_0/\mu}$, $x_2 =  (j ac_0d  G^2)/\mu$ and $x_3 = caG^2/\mu$. Similarly to Inequality~\eqref{eq:inductionplus}, we have
\begin{align*}
  \Reg_{j+1}(\theta) 
    & \leq \max\big\{ 3 x_1^2/4, (\Reg_j(\theta))_+ + x_1\sqrt{x_2}\big\}  + x_3  \\
    & = \max\left\{48 \frac{a d_0 c^2 G^2}{\mu}, (\Reg_j(\theta))_+ + \frac{8c\sqrt{c_0}a \sqrt{d_0d} G^2 }{\mu}\right\}  + \frac{caG^2}{\mu} \\
    & \leq (\Reg_j(\theta))_+ + \frac{(49 + 8c\sqrt{c_0}) a \sqrt{d_0d} G^2 }{\mu} \,.
\end{align*}
Choosing $c_0>0$ such that $49 + 8c\sqrt{c_0} \leq c_0$  concludes the induction. In particular, considering $j=I = \log_2(T-1)$, we proved that 
\[
  R_T(\theta) \leq \cO \left( \frac{ \sqrt{d_0 d} G^2 (\log d + \log \log (ET)) \log T }{\mu T} \right) \,.
\]

\subsection{Proof of Theorem~\ref{thm:acceleratedBOA}}
\label{app:proof_SABOA}
    We recall that $ \Theta^\ast = \argmin_{ \theta \in \cB_1} \E[\ell_t( \theta)]$. The idea of the proof is to show that at each session $i$, SABOA performs BOA by adding  sparse estimators  in $\Theta^{(i)}$ that are exponentially closer to  $\Theta^\ast$.

    Let $x >0$. We prove by induction on $i \geq 0$ that with probability at least $1 - i e^{-x}$, it exists $\theta^\ast \in \Theta^*$ such that
    \begin{equation}
        \tag{$\mathcal{H}_i$}
        \label{eq:induction}
        D\big(\theta^*,\Theta^{(i)}\big)  \leq    \epsilon 2^{- \tau i} \,,
    \end{equation}
    where $D$ is defined in Definition~\ref{def:distance},
    \begin{equation}
    \epsilon := \max_{\theta^*\in\Theta^*}\left((8\sqrt{a}G )^\beta \max\left\{\frac{2}{\alpha G^2} , \frac{8\|\theta^*\|_0}{\mu}\min \Big\{
      \frac{8\|\theta^*\|_0}{\|\theta^*\|_1^2 }, 
      \frac{1}{(1-\|\theta^*\|_1)^2}
    \Big\} \right\}\right)^{\frac{1}{2-\beta}},
    \label{eq:choiceepsilon}
    \end{equation}
    and $\tau = \frac{1}{2-\beta} - \frac{1}{2}$. Remark that $\theta^*$ in \eqref{eq:induction} depends on $i$ when $\Theta^*$ is not a singleton.

    \emph{Initialization}. For $i=0$, by definition (see Algorithm~\ref{alg:SABOA}), $\Theta^{(0)} := \{0\}$ and $D(\theta^*,\{0\})\leq \|\theta^*\|_1\leq 1$. The initialization thus holds true as soon as $\epsilon > 1$.

    \emph{Induction step}. Let $i \geq 0$ and assume \eqref{eq:induction}.
     We start from Theorem~\ref{thm:slowrate} (see Inequality~\eqref{eq:thmslowrateexact} for the precise constants that we upper-bound here) that we apply for $t=t_{i-1},\dots,t_i-1$ and $\theta^*\in \Theta^*$ satisfying \eqref{eq:induction}: with probability $1-e^{-x}$
\[
  \frac{1}{2^{i-1}} \sum_{t=t_{i-1}}^{t_i-1} \E_{t-1}\big[\ell_t(\hat \theta_{t-1}) - \ell_t(\theta^*)\big]
  \le \frac{2 \sqrt{a} G D(\theta^*,\Theta^{(i)}) }{2^{(i-1)/2}}  + 4 \left(\frac{a}{\alpha 2^{i-1}} \right)^{\frac{1}{2-\beta}} + \frac{aE}{2^{i-1}} +\frac{2}{2^{2i-2}} \,,
\]
where for simplicity of notation we define
$a := 16 (1+\log (K_i)) + 16 \log \log (ET^2)+ 4x$, where $K_i := \Card(\Theta^{(i)}) + 2d$ denotes the number of experts used during the doubling session $i$, and where we used $t_i = t_{i-1} + 2^{i-1}$.
Using \eqref{eq:induction} together with Jensen's inequality and recalling $\bar \theta^{(i)} := 2^{-i+1} \sum_{t=t_{i-1}}^{t_i-1} \hat \theta_{t-1}$, we obtain
  \begin{equation}
    \E\big[\ell_t(\bar \theta^{(i)}) - \ell_t( \theta^\ast)\big] \leq  2\sqrt{2a} G  \epsilon  2^{-(\frac{1}{2}+\tau)i} + 4 \Big(\frac{a}{\alpha}\Big)^{\frac{1}{2-\beta}} 2^{-\frac{i}{2-\beta}}  + {aE2^{1-i}} + {2^{3-2i}} \,.
    \label{eq:compliquee}
  \end{equation}
  Now, we simplify this expression by showing that the last three terms of the right-hand side are negligible with respect to the first one. First, because $a\geq 16$ and $E\geq 1$, we have $16 \leq aE$ and thus  $2^{3-2i} \leq aE 2^{-1-i}$. Then, because $\epsilon \geq \sqrt{a}$, $aE \leq \sqrt{a}E\epsilon = \frac{4}{3} \sqrt{a}\epsilon G$ and thus
  \begin{equation}
      {aE2^{1-i}} + {2^{3-2i}} \leq \frac{3}{2}aE2^{-i} \leq 2\sqrt{a}\epsilon G 2^{-i} \stackrel{\tau \leq 1/2}{\leq} 2\sqrt{a}\epsilon G 2^{-(\frac{1}{2}+\tau)i} \,.
      \label{eq:dominated1}
  \end{equation}
  The second term is also dominated thanks to the definition of $\epsilon$ in~\eqref{eq:choiceepsilon}
  \[
    \epsilon \stackrel{\eqref{eq:choiceepsilon}}{\geq}\left(\frac{2(\sqrt{8a}G )^\beta}{\alpha G ^2}\right)^{\frac{1}{2-\beta}} \quad \Rightarrow \quad 2\sqrt{2a}G\epsilon \geq  \Big(\frac{16a}{\alpha}\Big)^{\frac{1}{2-\beta}} \stackrel{0\leq \beta\leq 1}{\geq} 4 \Big(\frac{a}{\alpha}\Big)^{\frac{1}{2-\beta}}
  \]
  and 
  \[
    \tau  \stackrel{\eqref{eq:choiceepsilon}}{=} \frac{1}{2-\beta} - \frac{1}{2} \quad \Rightarrow \frac{1}{2-\beta} \geq \frac{1}{2} + \tau 
  \]
  which yields
  \begin{equation}
      4 \Big(\frac{a}{\alpha}\Big)^{\frac{1}{2-\beta}} 2^{-\frac{i}{2-\beta}} \leq  2\sqrt{2a} G  \epsilon  2^{-(\frac{1}{2}+\tau)i}\,.
      \label{eq:dominated2}
  \end{equation}
  Thus replacing Inequalities~\eqref{eq:dominated1} and~\eqref{eq:dominated2} into Inequality~\eqref{eq:compliquee} and upper-bounding $4\sqrt{2}+2\leq 8$, we get for any $\theta^* \in \Theta^*$
  \begin{equation}
    \E\big[\ell_t(\bar \theta^{(i)}) - \ell_t( \theta^\ast)\big] \leq    8 \sqrt{a} G  \epsilon  2^{-(\frac{1}{2}+\tau)i}  \,.
    \label{eq:boundexcessloss}
  \end{equation}
  Using Assumption~\ref{ass:Lojasiewicz}, there exists at least one $\theta^* \in \Theta^*$ (which can be different from the preceding session), which satisfies
  \begin{equation}
     \big\|\bar \theta^{(i)} -  \theta^\ast \big\|_{\infty} \leq \big\|\bar \theta^{(i)} -  \theta^\ast \big\|_{2} \stackrel{\eqref{eq:boundexcessloss}+\ref{ass:Lojasiewicz}}{\leq}  (8\sqrt{a} G  \epsilon)^{\frac{\beta}{2}}  \mu^{-\frac{1}{2}}  2^{-(\frac{1}{2}+\tau)\frac{\beta}{2} i} =: \epsilon' \,.
     \label{eq:bartheta}
  \end{equation}
  Now, we want to apply Lemma~\ref{lem:softthreshold} if $\|\theta^*\|_1$ is close to 1 and Lemma~\ref{lem:rlarge} if $\|\theta^*\|_1<1$. 
  In order to apply Lemma~\ref{lem:rlarge}, we consider  hard-truncated estimators $[\bar \theta^{(i)}]_{\tilde d_0}$, canceling the $d-\tilde d_0$ smallest components  of $\bar \theta^{(i)}$ for $\tilde d_0 \in \{1,\dots, d\}$. For the (unknown) choice $\tilde d_0 = d_0$, since $\|[\bar \theta^{(i)}]_{d_0}\|_0 = \|\theta^*\|_0 = d_0$, we have $\|[\bar \theta^{(i)}]_{d_0}-\theta^*\|_0 \leq 2d_0$ and
  \begin{multline*}
    \big\|[\bar \theta^{(i)}]_{d_0} -  \theta^\ast \big\|_{1} \leq \sqrt{2d_0}\big\|[\bar \theta^{(i)}]_{d_0} -  \theta^\ast \big\|_{2} \leq \sqrt{2d_0} \big(\big\|[\bar \theta^{(i)}]_{d_0} - \bar \theta^{(i)}\big\|_2 + \big\|\bar \theta^{(i)} -  \theta^\ast \big\|_{2}\big) \\
     \leq 2\sqrt{2d_0} \big\|\bar \theta^{(i)} -  \theta^\ast \big\|_{2} \leq 2\sqrt{2d_0}\epsilon' \,.
  \end{multline*}
  Applying Lemma~\ref{lem:rlarge}, we get
  \begin{equation}
    D\big(\theta^*, [\bar \theta^{(i)}]_{d_0}\big) \leq \frac{\big\|[\bar \theta^{(i)}]_{d_0} -  \theta^\ast \big\|_{1}}{1-\|\theta^*\|_1} \leq \frac{2\sqrt{2d_0}\epsilon'}{1-\|\theta^*\|_1} \,.
    \label{eq:D1smallnorm}
  \end{equation}
  This bound is only useful for $\|\theta^*\|_1 <1$. Otherwise, we want to apply Lemma~\ref{lem:softthreshold}.  However the values of $\epsilon'$ and $d_0 = \|\theta^*\|_0$ are unknown. We approximate them with $\tilde \epsilon$ and $\tilde d_0$ on exponential grids, which we define now:
  \[
      \cG_{\epsilon'} = \big\{2^{-k},\quad k=0,\dots,i\big\} \quad \text{and} \quad \cG_{d_0} = \big\{1,2,\dots 2^{-\lfloor \log d\rfloor},d\big\} \,.
  \]
  We define for all $\tilde \epsilon \in \cG_{\epsilon'}$ and $\tilde d_0 \in \cG_{d_0}$ the dilated soft-threshold
  \begin{equation}
    \tilde \theta(\tilde \epsilon,\tilde d_0) := S_{\tilde \epsilon}(\bar \theta^{(i)}) \left(1+\frac{2\tilde d_0\tilde \epsilon}{\|S_{\tilde \epsilon}(\bar \theta^{(i)})\|_1}\right)\wedge \frac{1}{\|S_{\tilde \epsilon}(\bar \theta^{(i)})\|_1} \,,
    \label{eq:defthetatilde}
  \end{equation}
  with the convention $\frac00=0$, recalling the definition of the soft-threshold operator $S_\epsilon(x)_i=\sign(x_i)(|x_i|-\epsilon)_+$ for all $1\le i\le d$.
  Because $\epsilon' \geq 2^{-i}$ (using $\epsilon \geq \sqrt{a} \geq 4$, $G\geq 1$ and $\tau \leq 1/2$ and $\mu\ge1$) and $\big\|\bar \theta^{(i)} -  \theta^\ast \big\|_{\infty} \leq 1$, it exists $\tilde \epsilon \in \cG_{\epsilon'}$ such that $\tilde \epsilon \leq 2\epsilon'$ and $\big\|\bar \theta^{(i)} -  \theta^\ast \big\|_{\infty} \leq \tilde \epsilon$. Furthermore, it exists also $\tilde d_0 \in \cG_{d_0}$ such that $d_0 \leq \tilde d_0 \leq 2d_0$. We can thus apply Lemma~\ref{lem:softthreshold}, which yields
  \begin{equation}
    D(\theta^*,\tilde \theta(\tilde \epsilon,\tilde d_0)\big) \leq \frac{2\tilde d_0 \tilde \epsilon}{\|\theta^*\|_1} \leq \frac{8 d_0  \epsilon'}{\|\theta^*\|_1}  \,.
    \label{eq:D1largenorm}
  \end{equation}
  We define the new approximation grid 
  \begin{equation}
    \Theta^{(i+1)} := \big\{\tilde \theta(\tilde \epsilon,\tilde d_0), \tilde \epsilon \in \cG_{\epsilon'}, \tilde d_0 \in \cG_{d_0}\big\} \cup \{[\bar\theta^{(i)}]_{\tilde d_0},\quad \tilde d_0 = 1,\dots, d\big\} \,,
    \label{eq:defgrid}
  \end{equation}
  where $\tilde \theta(\tilde \epsilon,\tilde d_0)$ is defined in Equation~\eqref{eq:defthetatilde} and $[\cdot]_k$ are hard-truncations to $k$ coordinates. We get from Inequality~\eqref{eq:D1smallnorm} and~\eqref{eq:D1largenorm} that 
  \begin{align*}
    D\big(\theta^*,\Theta^{(i+1)}\big) 
      & \leq \min\left\{ \frac{\sqrt{8d_0}}{1-\|\theta^*\|_1}, \frac{8d_0}{\|\theta^*\|_1} \right\}\epsilon' \\
      & \stackrel{\eqref{eq:bartheta}}{=}  (8\sqrt{a} G  \epsilon)^{\frac{\beta}{2}}  \mu^{-\frac{1}{2}} \min\left\{ \frac{\sqrt{8d_0}}{1-\|\theta^*\|_1}, \frac{8d_0}{\|\theta^*\|_1} \right\}  2^{-(\frac{1}{2}+\tau)\frac{\beta}{2} i} \,.
  \end{align*}

  To conclude the induction, it suffices to show that this is smaller then $\epsilon 2^{-\tau (i+1)}$.  Our choices of $\epsilon$ and $\tau$ defined in~\eqref{eq:choiceepsilon} was done in that purpose, so that the induction is completed.

  \emph{Conclusion}. Substituting the values of $\epsilon$ and $\tau$ into Inequality~\eqref{eq:boundexcessloss} and using the choice $i = \log_2 T$ (which upper-bound the number of sessions after $T$ times steps) concludes the proof:
\begin{align*}
    \E\big[\ell_t(\bar \theta^{(i)}) - \ell_t( \theta^\ast) \big]
       & \stackrel{\text{Jensen}}{\leq}  \frac{R_T^{(i)}}{2^i}  \\
       & \stackrel{\eqref{eq:boundexcessloss}}{\leq} 8 \sqrt{a} G  \epsilon  2^{-(\frac{1}{2}+\tau)i} \\
       & \stackrel{\eqref{eq:choiceepsilon}}{\leq} \max_{\theta^*\in\Theta^*} \left(\frac{128 a}{T} \max\left\{\frac{1}{\alpha}, \frac{4G^2\|\theta^*\|_0}{\mu} \min\left\{\frac{1}{(1-\|\theta^*\|_1)^2},\frac{8\|\theta^*\|_0}{\|\theta^*\|_1^2} \right\} \right\}\right)^{\frac{1}{2-\beta}} \,,
\end{align*}
where we recall that $a := 16 (1+\log (K_i) + \log \log (ET^2))+ 4x$, where $K_i := \Card(\Theta^{(i)}) + 2d \leq (1+\log_2 d)(1+\log_2 T) + d $.
  Summing over $i=1,\dots,\log_2(T)$, we get the upper-bound for the cumulative risk.

\end{document}